\documentclass[11pt,a4paper]{article}
\usepackage{fancyhdr}
\usepackage[british]{babel}

\usepackage{amssymb, amsmath, amsthm, amsfonts, enumerate,bm, nicefrac}
\usepackage{amsmath,setspace,scalefnt}
\usepackage[usenames,dvipsnames,svgnames,table]{xcolor}
\usepackage{graphicx,tikz,caption,subcaption,color}
\usetikzlibrary{patterns}
\usetikzlibrary{shapes}
\usepackage{fullpage} 
\usepackage[colorlinks=true,anchorcolor=blue,filecolor=blue,linkcolor=red,urlcolor=blue,citecolor=blue]{hyperref}  

\usepackage{enumitem,verbatim}
\usepackage{multicol}
\usepackage{float}
\usepackage{cleveref}
\usepackage[margin=2.49cm]{geometry}

\definecolor{gray}{rgb}{0.25, 0.25, 0.25}

\usepackage{booktabs}  

\counterwithin{figure}{section}

\newtheorem{theorem}{Theorem}[section]
\newtheorem{lemma}[theorem]{Lemma}
\newtheorem{claim}[theorem]{Claim}
\newtheorem{cor}[theorem]{Corollary}

\newtheorem{conj}[theorem]{Conjecture} 

\newtheorem*{observation*}{Observation}

\newtheorem{problem}[theorem]{Problem}
\newtheorem*{problem*}{Problem}

\newtheorem*{question*}{Question}
\newtheorem{innerdefinition}{Definition}
\newenvironment{definition*}
  {
   \innerdefinition}
  {\endinnerdefinition}

\theoremstyle{definition}

\newtheorem*{defn*}{Definition}
\newtheorem{example}[theorem]{Example}

\theoremstyle{remark}

\newenvironment{poc}{\begin{proof}[Proof of claim]}{\end{proof}}

\usepackage{todonotes}

\newcommand*{\abs}[1]{\lvert#1\rvert}

\title{Edge-spectral Tur\'{a}n theorems for color-critical graphs with applications\footnote{An earlier version of this paper was announced in November 2025 on arXiv. This is the final version, and it has been accepted for publication in SIAM Journal on Discrete Mathematics. Any comments are welcome.}}

\author{
Yongtao Li\thanks{Yau Mathematical Sciences Center, Tsinghua University, Beijing, China. Email: \url{ytli0921@hnu.edu.cn}. Supported by the Postdoctoral 
Science Foundation of China (No. 2026M793384) and 
the Postdoctoral Fellowship Program of China Postdoctoral Science Foundation (No. GZC20261656).}
\and
Hong Liu\thanks{Extremal Combinatorics and Probability Group (ECOPRO), Institute for Basic Science (IBS), Daejeon, South Korea. Email: \url{hongliu@ibs.re.kr}. Supported by the Institute for Basic Science (IBS-R029-C4).}
\and
Shengtong Zhang\thanks{Department of Mathematics, Stanford University, CA, USA. Email: \url{stzh1555@stanford.edu}. Supported by the Craig Franklin Fellowship in Mathematics at Stanford University.}
}
\date{}

\begin{document}
\maketitle

\begin{abstract}  
A celebrated theorem due to Nikiforov \cite{Niki2002} states that if $G$ is a $K_{r+1}$-free graph with $m$ edges, then $\lambda^2 (G)\le (1-\frac{1}{r})2m$. This result not only implies both the Tur\'{a}n and Wilf theorems, but also offers a new perspective on the existence of substructures. The edge-spectral assumptions are particularly versatile for enforcing substructures, since they can be applied to any sparse graph regardless of its edge density. In this paper, we prove that 
for any color-critical graph $F$ with chromatic number $\chi (F)=r+1\ge 4$, if 
$m$ is sufficiently large and $G$ is an $F$-free graph with $m$ edges, then $\lambda^2 (G)\le (1-\frac{1}{r})2m$, with equality if and only if $G$ is a regular complete $r$-partite graph. This settles an open problem proposed by Yu and Li \cite{YL2025arxiv} and also gives spectral bounds for graphs forbidding books and wheels. 

We also establish an asymptotic formula and structural characterization when we forbid an almost-bipartite graph $F$, where $F$ is called almost-bipartite if it can be made bipartite by removing at most one edge. As applications, we determine the unique $m$-edge spectral extremal graph for all sufficiently large integers $m$ when avoiding certain substructures, including complete bipartite graphs plus an edge, cycles plus an edge, and theta graphs. Our results resolve an open problem proposed by Li, Zhao and Zou \cite{LZZ2024}, as well as two conjectures posed by Liu and Li \cite{LL2025laa}. 
The arguments in our proofs are based on the edge-spectral stability method recently established by the authors, and a stability result for the Perron--Frobenius eigenvector, in the spirit of the Davis--Kahan perturbation bound.
\end{abstract}




\section{Introduction}

Extremal graph theory, a central branch of graph theory, examines the maximum possible size or structure of graphs under certain constraints. 
The field has deep ties to combinatorics, computer science, and additive number theory; see \cite{GTAC}. A graph $G$ is called {\it $F$-free}  if it does not contain a subgraph isomorphic to $F$. 
 The {\em Tur\'{a}n number} of $F$, denoted by $\mathrm{ex}(n, F)$, is defined as the maximum number of edges
  in an $n$-vertex $F$-free graph.
  An $F$-free graph with the maximum number of edges is called an {\em extremal graph} for $F$. 
 The Tur\'{a}n graph $T_{n,r}$ is a complete $r$-partite graph on $n$ vertices whose part sizes are as equal as possible. 
 The well-known Tur\'{a}n theorem states that
if $G$ is an $n$-vertex $K_{r+1}$-free graph,
then $e(G)\le e(T_{n,r})$, 
with equality if and only if $G=T_{n,r}$; in particular, $e(G)\le \left(1-\frac{1}{r} \right) \frac{n^2}{2}.$ 
A celebrated extension of Tur\'{a}n's theorem 
is attributed to Erd\H{o}s, Stone and Simonovits \cite{ES46,ES66}.

\begin{theorem}[Erd\H{o}s--Stone--Simonovits \cite{ES46,ES66}]
  \label{thm-ESS} 
  For any graph $F$ with chromatic number $\chi (F)=r+1\ge 2$, 
 if $G$ is an $F$-free graph on $n$ vertices, then 
$ e(G)\le \left( 1- \frac{1}{r} + o(1) \right)\frac{n^2}{2}$.  
\end{theorem}
  
The {\it adjacency matrix} of an $n$-vertex graph 
$G$ is defined as $A_G=[a_{i,j}]_{i,j=1}^n$, 
where $a_{i,j}=1$ if $ij\in E(G)$, and $a_{i,j}=0$ otherwise. 
The {\it spectral radius} $\lambda (G)$ is defined as the maximum modulus of eigenvalues of $A_G$. 
 Since $A_G$ is a non-negative matrix,  
the Perron--Frobenius theorem implies that $\lambda (G)$ is in fact the largest eigenvalue of $A_G$.  
The investigation of spectral graph theory significantly contributes to advancements in graph theory and combinatorics; 
see \cite{Huang2019,JTYZZ2021} for recent breakthroughs. 
Spectral extremal graph theory studies the connections between the structure of extremal graphs and the eigenvalues of their associated matrices. This research area has led to natural spectral strengthenings of classical extremal graph theorems. 
In 1986,  Wilf \cite{Wil1986} proved that for every $n$-vertex $K_{r+1}$-free graph $G$, we have 
$\lambda (G)\le \left(1-\frac{1}{r} \right)n.$ 
In 2009, Nikiforov \cite{Niki2009cpc} showed a spectral extension of Theorem \ref{thm-ESS}.

 \begin{theorem}[Nikiforov \cite{Niki2009cpc}] 
 \label{thm-spec-ESS} 
For any graph $F$ with chromatic number $\chi (F)=r+1\ge 2$, 
if $G$ is an $F$-free graph on $n$ vertices, then 
$ \lambda (G)\le \left( 1-\frac{1}{r} + o(1)\right) n.$
\end{theorem}

\subsection{Edge-spectral extremal graph problem}

In this paper, we study the spectral extremal results for graphs with a given number of edges, rather than a given number of vertices. For simplicity, we usually ignore the isolated vertices of $G$ when this causes no confusion. It is well-known that $\lambda (G)< \sqrt{2m}$ for every graph $G$ with $m$ edges.  
There is a rich history to the problem of characterizing graphs that attain the maximum spectral radius for a given number of edges. 
 The story goes back to 1985, when Brualdi and Hoffman \cite{BH1985} proved that if $G$ has $m\le {k \choose 2}$ edges for some integer $k\ge 2$, then $\lambda (G)\le k-1$. 
 Stanley \cite{Sta1987} extended this result and showed that $\lambda (G)\le \frac{1}{2} 
 \left(\sqrt{8m+1} -1 \right)$, with equality if and only if $m={k \choose 2}$ and $G$ is the complete graph $K_k$. 
We refer to \cite{Hong1988,HSF2001} for related results.   

The above-mentioned bounds can be significantly improved for graphs with bounded clique number.  
In 2002, Nikiforov \cite{Niki2002} showed that if $G$ is a $K_{r+1}$-free graph with $m$ edges, then 
\begin{equation}
    \label{eq-Niki-2002-cpc}
    \lambda^2 (G)\le {\left(1- \frac{1}{r} \right)2m}.
\end{equation}
The extremal graphs attaining equality in  (\ref{eq-Niki-2002-cpc}) are the complete bipartite graphs when $r=2$, and the regular complete $r$-partite graphs when $r\ge 3$; see \cite{Niki2006-walks,Niki2009jctb} for details.

The result (\ref{eq-Niki-2002-cpc}) implies both Tur\'{a}n's bound and Wilf's bound by using $\lambda (G)\ge \frac{2m}{n}$. So (\ref{eq-Niki-2002-cpc}) can be viewed as a robust spectral extension of Tur\'{a}n's theorem. 
We note that the edge-spectral assumption offers greater applicability and versatility: the edge-spectral version not only implies the vertex-spectral version, but can also be applied to sparse graphs of any edge density. For example, in the case $r=2$, setting $G=K_2\vee \frac{m-1}{2}K_1$, the join of $K_2$ with $\frac{m-1}{2}$ isolated vertices, we see that $G$ is a sparse graph with $\lambda (G) > \sqrt{m}$ and $G$ has many triangles, but $\lambda (G) \ll n/2$ and $e(G) \ll n^2/4$.  
In light of the above difference and advantage, there has been great interest in studying the maximum spectral radius of $F$-free graphs with a given number of edges.

In 2009, Nikiforov \cite{Niki2009} showed that if 
$G$ is an $m$-edge graph with $\lambda (G)> \sqrt{m}$, then $G$ contains not only a triangle, but also a $4$-cycle. In 2023, Ning and Zhai \cite{NZ2021} studied the spectral supersaturation for triangles and proved that such a graph $G$ has at least $\lfloor \frac{1}{2} (\sqrt{m}-1)\rfloor$ triangles. Solving a problem of Ning and Zhai \cite{NZ2021b}, the authors \cite{LLZ2024-book-qua} proved that $G$ also has at least $(\frac{1}{8}-o(1))m^2$ copies of $C_4$. 
Moreover, the authors \cite{LLZ2024-book-qua} showed that if $\lambda^2(G) > (1-\frac{1}{r})2m$, then there are $\Omega_r(m^{\frac{r-1}{2}})$ copies of $K_{r+1}$ in $G$. 
Recently, the authors \cite{LLZ2025-part-1} established an asymptotic formula for any general  graph $F$, and provided a unified extension of Theorems \ref{thm-ESS} and \ref{thm-spec-ESS}.

\begin{theorem}[Li--Liu--Zhang \cite{LLZ2025-part-1}] \label{thm:LLZ2025}
For any graph $F$ with chromatic number $\chi (F) = r+1\ge 3$, if $G$ is an $F$-free graph with $m$ edges, then 
$\lambda^2 (G)\le \left(1-\frac{1}{r} + o(1)\right)2m.$
\end{theorem}

 Motivated by the result of Brualdi and Hoffman, 
 we consider the following spectral Tur\'{a}n type extremal problem for $F$-free graphs with $m$ edges. 

\begin{problem}[Brualdi--Hoffman--Tur\'{a}n type problem]  \label{problem-BHT}
Given a graph $F$ and an integer $m$, 
   what is the maximum spectral radius over all $F$-free graphs with $m$ edges? 
\end{problem}

Throughout this paper, we always assume that $F$ is a graph with no isolated vertices.    
The investigation of Problem \ref{problem-BHT} has become an active research topic in spectral extremal graph theory. In recent years, there have been a large number of results on this problem for some forbidden graphs. We refer the interested reader to 
the results on cycles \cite{ZLS2021,LZS2024,LLLY2025}, friendship graphs \cite{LLP2023,CY2024}, fan graphs \cite{YLP2024,LZZ2024} and books \cite{Niki2021,LLZ2024-book-qua, ZLL2026}, and to  the references therein.

The previously known approaches to Problem \ref{problem-BHT} rely primarily on the second moment of the spectral radius and on the $k$-core method. However, their applicability is limited to forbidding a graph $F$ with chromatic number $\chi (F)\le 3$.  
The main contribution of this paper is to overcome this limitation by developing a new technique, which we term {\it the edge-spectral stability method}.

\subsection{Color-critical graphs}

 A graph $F$ is called {\it color-critical} if it contains an edge whose deletion reduces its chromatic number. Many graphs are color-critical, including  complete graphs and odd cycles.  
 There has been extensive research on extremal problems concerning color-critical graphs~\cite{Sim1966,Mubayi2010,PY2017,RS2018,ZL2022jgt,LL2024-ejc}. 
 Simonovits \cite{Sim1966} proved that 
 if $F$ is color-critical and $\chi (F)=r+1$, 
then the Tur\'{a}n graph $T_{n,r}$ is the unique extremal graph for $F$ when $n$ is sufficiently large. Nikiforov \cite{Niki2009ejc} provided a spectral extension by showing that 
$T_{n,r}$ is also the unique spectral extremal graph for $F$. 

As the first main result, we considerably extend the celebrated Nikiforov inequality (\ref{eq-Niki-2002-cpc}) and completely solve Problem \ref{problem-BHT} 
for all color-critical graphs $F$ with $\chi (F)\ge 4$.   
We obtain an edge-spectral counterpart of the 
theorems of Simonovits and Nikiforov mentioned above. 
Apart from the clique case in (\ref{eq-Niki-2002-cpc}), our result provides the first solution to Problem \ref{problem-BHT} when $\chi (F)\ge 4$.

\begin{theorem}
    \label{thm:color-critical-I}
    Let $F$ be a color-critical graph with $\chi (F)=r+1\ge 4$. For sufficiently large $m$, if $G$ is an $F$-free graph with $m$ edges, then
    $$\lambda^2 (G)\le \left(1- \frac{1}{r}\right) 2m,$$
    with equality if and only if $G$ is a regular complete $r$-partite graph.
\end{theorem} 

Theorem \ref{thm:color-critical-I} settles an open problem proposed by Yu and Li \cite[Problem 5]{YL2025arxiv}. 
In addition, by invoking the fact $\lambda (G) \ge \frac{2}{n}e(G)$, we see that Theorem \ref{thm:color-critical-I} implies that 
if $n$ is sufficiently large and $G$ is an $F$-free graph on $n$ vertices, then $e(G)\le (1-\frac{1}{r})\frac{n^2}{2}$ and $\lambda (G)\le (1-\frac{1}{r})n$. This recovers the previous bounds obtained by Simonovits and Nikiforov, respectively. 

\medskip 
\noindent 
{\bf Remark.}  
The bound of Theorem \ref{thm:color-critical-I} does not hold for color-critical graphs $F$ with $\chi (F)=3$. 
 For example, take $F=C_5$ and $G=K_2 \vee \frac{m-1}{2}K_1$. Then $G$ is $C_5$-free, 
but $\lambda (G)= \frac{1}{2} (1+\sqrt{4m-3}\,)> \sqrt{m}$. 
In addition, let $F=B_k$ be the book obtained from $k$ triangles sharing an edge.  
Nikiforov \cite{Niki2021} showed that every $B_k$-free graph $G$ satisfies $\lambda (G)\le \sqrt{m}$, with equality if and only if $G$ is a complete bipartite graph, which need not be regular. The case $\chi(F) \le 3$ is the subject of our next subsection. 

\medskip

The generalized book $B_{r,k}=K_r \vee I_k$ 
is a graph obtained from $k$ copies of 
$K_{r+1}$ sharing a common $K_r$. 
The case $r=2$ was studied by Nikiforov \cite{Niki2021} and the authors \cite{LLZ2024-book-qua}. 
For $r\ge 3$, we see that $B_{r,k}$ is color-critical and $\chi (B_{r,k})\ge 4$. 
As applications, Theorem \ref{thm:color-critical-I} 
implies the following result, which gives an alternative proof of a conjecture of Li, Liu and Feng \cite[Conjecture 1.20]{LLF2022}.

\begin{cor} \label{thm-Brk}
 Let $r\!\ge\! 3, k\!\ge\! 1\!$ be fixed and $m\!$ be large enough. 
If $G$ is an $m$-edge $B_{r,k}$-free graph, 
then  
$ \lambda^2 (G)\!\le \! \left( 1-\frac{1}{r} \right) 2m$,   
with equality if and only if $G$ is a regular complete $r$-partite graph. 
\end{cor}

 Li, Liu and Zhai \cite{LLZ2025-wheels} studied the edge-spectral problem for graphs forbidding the family of all wheels. 
As an application of Theorem \ref{thm:color-critical-I}, we weaken the condition by forbidding a single wheel. We denote $W_{2k+2}=K_1\vee C_{2k+1}$. 
Note that $W_{2k+2}$ is color-critical and $\chi (W_{2k+2})=4$.  

\begin{cor} \label{thm-even-wheel}
For fixed $k$ and sufficiently large $m$, 
if $G$ is a $W_{2k+2}$-free  graph with $m$ edges, then 
$ \lambda (G) \le \sqrt{4m/3}$,   
with equality if and only if $G$ is a regular complete $3$-partite graph. 
\end{cor}

Corollary \ref{thm-even-wheel} confirms a conjecture of Yu, Li and Peng \cite[Conjecture 5.3]{YLP2024}.

\subsection{Almost-bipartite graphs}

A graph $F$ is called {\it almost-bipartite} if it can be made bipartite by removing at most one edge. Equivalently, $F$ is either bipartite, or color-critical with $\chi (F)=3$. 
Using the edge-spectral stability method, we establish an asymptotic formula for Problem \ref{problem-BHT} when we forbid an almost-bipartite graph $F$. 
As applications, we completely determine the $m$-edge spectral extremal graphs for forbidding classical graphs, including complete bipartite graphs plus an edge, cycles plus an edge, and theta graphs. 
Consequently, we generalize a series of results in \cite{ZLS2021,LZS2024,LZZ2024,LLLY2025}, and settle several problems and conjectures in \cite{LZZ2024,LL2025laa}. 
In addition, as special cases of our theorems, 
we unify some recent results from \cite{FKW2024,FYH2022,GL2025-Mar,LL2025laa,LW2022,MLH2022,LLL2024,LSW2022,WW2025} and moreover we apply our theorems to prove several new results. 
This significantly extends our knowledge of the Brualdi--Hoffman--Tur\'{a}n problem. 
To state our results, we need to fix some notation. 

\begin{defn*} 
For an almost-bipartite graph $F$, let $\mathcal{A}_F$ be the family of subgraphs of $F$ induced by $\overline{I}:=V(F)\setminus I$, where $I$ ranges over the maximal independent sets of $F$. Let $\mathcal{M}_F$ be the family of graphs that do not contain a subgraph isomorphic to any member of $\mathcal{A}_F$.  
\end{defn*} 

The graphs in $\mathcal{A}_F$ and $\mathcal{M}_F$ may contain isolated vertices. 
For example, if $F=K_{s,t}$ with $s\le t$, 
then $\mathcal{A}_F=\{sK_1,tK_1\}$, and $\mathcal{M}_F$ consists of graphs of order less than $s$.   
If $F=K_{s,t}^+$, which is obtained from $K_{s,t}$ by adding an edge to the part of size $s$, 
then $\mathcal{A}_F=\{K_2\cup (s\!-\!2)K_1, K_{1,t}\}$, 
and $\mathcal{M}_F$ consists of empty graphs of any order and non-empty graphs of order less than $s$. 

In what follows, 
we show a structural characterization of extremal graphs for Problem \ref{problem-BHT} when we forbid an almost-bipartite graph $F$.  
An $m$-edge graph $G$ is called a {\it spectral extremal graph} for $F$ if $G$ achieves the maximum spectral radius over all $m$-edge $F$-free graphs. We say that a vertex $v$ is {\it complete} to a set $A$ if $v$ is adjacent to all vertices of $A$.

\begin{theorem}
    \label{thm:bipartite-rough}
    For any almost-bipartite graph $F$ that is not a star, 
    if $m$ is sufficiently large and $G$ is an $m$-edge spectral extremal graph for $F$,  
    then one of the following holds: 
    \begin{enumerate}
    \item[\rm (a)]
    $G$ is a complete bipartite graph. 

    \item[\rm (b)] 
    $G$ has a vertex partition $A\sqcup C$ such that $|A|< |F|$ and $G[A]$ is a non-empty graph of $ \mathcal{M}_F$, $C$ is an independent set, and all but at most one vertex of $C$ are complete to $A$.
    \end{enumerate}
\end{theorem}

\noindent 
\textbf{Remark.} 
If $F$ is a star $K_{1, t}$, then the $F$-free graph $G$ 
satisfies $\Delta (G)\le t-1$, which gives $\lambda (G)\le t-1$, with one of the extremal graphs being the disjoint union of copies of $K_{t}$.

\medskip

 Theorem \ref{thm:bipartite-rough} gives an alternative proof of a conjecture proposed by Zhai, Lin and Shu \cite[Conjecture 5.2]{ZLS2021} on books, which was confirmed by Nikiforov \cite{Niki2021} and the authors \cite{LLZ2024-book-qua} in a strong sense. 
Indeed, if $F$ is a book $B_k$, then $\mathcal{A}_F=\{K_2,K_{1,k}\}$ and  
each graph of $ \mathcal{M}_F$ is empty. Applying Theorem \ref{thm:bipartite-rough}, we see that 
the $m$-edge spectral extremal graphs for $B_k$ are complete bipartite graphs. 
Conversely, assume that the $m$-edge spectral extremal graphs for $F$ are complete bipartite graphs. Then the extremal spectral radius equals  $\sqrt{m}$, and hence  
$F$ is a subgraph of $B_k$ for some $k$. Otherwise, if $F$ is not a subgraph of any book, then $G:=K_{2}\vee  \frac{m-1}{2}K_1$ is an $F$-free graph with $\lambda(G) > \sqrt{m}$, a contradiction. So $F$ must be a subgraph of a book. 

\medskip

Our second main result is an asymptotic formula for almost-bipartite graphs. 

\begin{theorem} \label{thm-determine-M}
    Let $F$ be an almost-bipartite graph that is not a star. If $m$ is sufficiently large and 
   $G$ is an $m$-edge $F$-free graph with maximum spectral radius, then 
    $$\lambda(G) = \sqrt{m} + \max_{M \in \mathcal{M}_F} \frac{e(M)}{v(M)} + O(m^{-1/2}).$$
    Furthermore, there is a partition $V(G)=A\sqcup C$ such that $G[A]$ is isomorphic to $M$ that achieves the maximum above, $C$ is independent, and all but at most one vertex of $C$ are complete to $A$. 
 \end{theorem} 

The bound of Theorem \ref{thm:LLZ2025} does not hold for any bipartite graph $F$, unless $F$ is a star. 
The bound of Theorem \ref{thm:color-critical-I} does not hold for color-critical graphs $F$ with $\chi (F)=3$. 
To some extent, 
Theorem \ref{thm-determine-M} can be viewed as an intriguing supplement to Theorems \ref{thm:LLZ2025}  and \ref{thm:color-critical-I}. 

 {The term $e(M)/v(M)$ in Theorem~\ref{thm-determine-M} is a second-order term. Simonovits advocated the study of second-order terms of Tur\'an-type problems even when the first-order asymptotics are known. A classical instance is the theorem of Erd\H{o}s and Simonovits \cite{ES1971} on
the octahedron $K_{2,2,2}$, where the second-order term of $\mathrm{ex}(n,K_{2,2,2})$ is governed by a degenerate extremal problem related to $K_{2,2}$. 
Theorem~\ref{thm-determine-M} shows the same phenomenon in the edge-spectral setting: the main term $\sqrt{m}$ does not depend on $F$, and the forbidden graph enters only through the second-order term $\max_{M\in\mathcal{M}_F}e(M)/v(M)$, which by the following Example~\ref{exam-chall} can encode a difficult degenerate problem.}  
We present an example where determining $\max_{M \in \mathcal{M}_F} {e(M)}/{v(M)}$ reduces to a challenging graph problem. This shows that when we forbid a general bipartite graph $F$, the structure of the edge-spectral extremal graph in Theorem~\ref{thm-determine-M} can be quite non-trivial. 

\begin{example} \label{exam-chall}
Pick $s\ge 2$ and $n > 2s$. Let $F$ be the graph with vertex set $V = I_1 \sqcup I_2 \sqcup I_3 \sqcup I_4$, where $I_1 = I_4 = [n + 1 - s]$ and $I_2 = I_3 = [s]$, obtained by putting edges between every pair of vertices in $I_i \times I_{i + 1}$ for each $i \in \{1,2,3\}$. Then $\mathcal{A}_F=\{(n+1)K_1, K_{s,s}\}$. So $\mathcal{M}_F$ consists of all $K_{s, s}$-free graphs with at most $n$ vertices. Finding the maximizer $M \in \mathcal{M}_F$ reduces to determining the Tur\'{a}n number $\mathrm{ex}(n, K_{s, s})$, which is a well-known difficult open problem. 
\end{example}

Let $I$ be a maximum independent set of $F$ with  $\abs{\bar{I}} = k + 1$. If $F$ is an almost-bipartite graph, then $F[\bar{I}]$ contains at most one edge, and hence a non-empty graph $M$ lies in $\mathcal{M}_F$ if and only if $M$ has at most $k$ vertices. So $K_{k}$ is the unique subgraph $M\in \mathcal{M}_F$ that maximizes the ratio $e(M) / v(M)$ in Theorem \ref{thm-determine-M}. 
This observation will be applied in the next section.

\subsection{Applications}

\label{subsec:appli}

We present several applications of Theorem \ref{thm-determine-M} in this section. 
We solve the 
Brualdi--Hoffman--Tur\'{a}n problem for 
classical almost-bipartite graphs, including complete bipartite graphs plus an edge, cycles plus an edge, and theta graphs. 
Our results significantly generalize a series of theorems in \cite{ZLS2021,LZS2024,FKW2024,LZZ2024,LLLY2025}. 
The previous method gives a bound that can be achieved only for some special values of $m$. 
In this paper, we determine the unique spectral extremal graph for every large integer $m$.  
To start with, we define the edge-spectral extremal graph as follows.

\begin{defn*}[The split graph]
For each $1\le k<m$, let $r,t$ be non-negative  integers such that $m- {k \choose 2}= k t + r$, where $0\le r\le k-1$. 
Let $S_{k,m}$ be an $m$-edge graph 
obtained from $K_{k}\vee tK_1$ by adding an extra vertex that has exactly $r$ neighbors in the vertex set of $K_k$. 
\end{defn*} 

A computation shows 
$ \lambda (S_{k,m})
\le \frac{1}{2} (k-1 + \sqrt{4m -k^2+1})$, where the equality holds if and only if $r=0$, i.e., $m={k \choose 2} + kt$ for some integer $t\ge 0$. Moreover, one can check that $\lambda (S_{k,m}) \le \lambda (S_{\ell,m})$ provided that $1\le k\le \ell$, where $\ell$ is fixed and $m$ is sufficiently large.

\subsubsection{Complete bipartite
graphs plus an edge}

In 2021, 
 Zhai, Lin and Shu \cite{ZLS2021} proved that if $t\ge 2$, $m\ge 16t^2$ and $G$ is an $m$-edge $K_{2,t+1}$-free graph, then $\lambda (G)\le \sqrt{m}$, with equality if and only if $G$ is a star. In this paper, we greatly extend this result to all complete bipartite graphs. 
Let $K_{k+1,t+1}^+$ be the graph obtained from  $K_{k+1,t+1}$ by adding an edge to the vertex part of size $k+1$.

\begin{theorem} \label{thm-unify}  
Let $2 \le k\le t$ be fixed and $m$ be sufficiently large. 
If $G$ is a $K_{k+1,t+1}^+$-free graph with $m$ edges, then 
$\lambda(G) \le \lambda (S_{k,m})$,
where the equality holds if and only if $G=S_{k,m}$. 
\end{theorem}

\begin{proof}
We see that  $K_{k+1,t+1}^+$ is color-critical with chromatic number $\chi (K_{k+1,t+1}^+)=3$. 
        Setting $F=K_{k+1,t+1}^+$, 
    it is easy to verify that $\mathcal{A}_{F} = 
    \{K_2\cup (k-1)K_1, K_{1,t+1}\}$. 
    Since $k\ge 2$, every graph in $\mathcal{M}_{F}$ is either an empty graph of any order, or a non-empty graph with at most $k$ vertices. 
    By Theorem \ref{thm-determine-M}, taking $M=K_k$ yields the desired spectral extremal graph.  
\end{proof}

\noindent 
{\bf Remark.} 
The bound on $m$ is super-polynomial with respect to  $k$. Suppose on the contrary that Theorem \ref{thm-unify} holds for $m\ge Ck^p$ with some $C>0$ and $p\ge 2$. 
For any $\varepsilon \in (0,0.01)$, we take $\delta := \varepsilon^{10}C^{-1/p} >0$ and $k:=\delta m^{1/p}$. 
Then $m> Ck^p$ and $(1+\varepsilon) \sqrt{m} >  
\frac{1}{2}\big(k-1 + \sqrt{4m-k^2+1} \big) \ge \lambda (S_{k,m})$, which implies that any $m$-edge graph $G$ with $\lambda (G)\ge (1+\varepsilon )\sqrt{m}$ contains a copy of $K_{k+1,k+1}$ with $k\gg \log m$. This is impossible by a random construction; see \cite[Theorem 6.3]{LLZ2025-part-1}.

\subsubsection{Theta graphs}

Nikiforov \cite{Niki2009} studied 
the Brualdi--Hoffman--Tur\'{a}n problem for $C_4$ and 
proved that if $m\ge 10$ and $G$ is a $C_4$-free graph with $m$ edges, then $\lambda (G)\le \sqrt{m}$, with equality if and only if $G$ is a star. 
Zhai, Lin and Shu \cite{ZLS2021} studied the spectral problem for $C_5$-free and $C_6$-free graphs. 
Let $C_k^+$ be the graph obtained from $C_k$ by adding an edge between two vertices of distance two. In 2024, Li, Zhai and Shu \cite{LZS2024} introduced the spectral $k$-core method and proved that for $k\ge 3$ and $m\ge 4(k^2 +3k +1)^2$, if $G$ is a $C_{2k+1}^+$-free or $C_{2k+2}^+$-free graph with $m$ edges, then 
\begin{equation} 
\label{thm-LZS2024} 
\lambda (G)\le \frac{k-1 +\sqrt{4m -k^2+1}}{2}, 
\end{equation}
where the equality holds if and only if $G=K_k \vee (\frac{m}{k} - \frac{k-1}{2})K_{1}$. 
This confirms a famous conjecture for cycles  suggested by Zhai, Lin and Shu \cite[Conjecture 5.1]{ZLS2021}, which was also selected by Liu and Ning \cite[Sec. 5]{LN2023-unsolved} as one of the ``Unsolved Problems'' in spectral graph theory.

Let $\theta_{r,p,q}$ be the theta graph consisting of  three internally disjoint paths of lengths $r,p,q\in \mathbb{N}$ with a common pair of endpoints.  
In 2025, Li, Zhao and Zou \cite{LZZ2024} extended the Li--Zhai--Shu result and showed further that the bound (\ref{thm-LZS2024}) still holds when $G$ is $\theta_{1,p,q}$-free, where $p+q\in \{2k+1,2k+2\}$ with some integer $k\ge 3$. 
Furthermore, Li, Zhao and Zou proposed a problem for the unsolved case  $\theta_{1,3,3}$ and a general problem for $\theta_{r,p,q}$ for every $r\ge 2$.

Observe that the bound (\ref{thm-LZS2024}) in the above results is attained only when $G=K_k \vee (\frac{m}{k} - \frac{k-1}{2})K_{1}$, that is, for integers $m$ such that $m- {k \choose 2}$ is divisible by $k$. 
However, for other values of $m$, these results do not characterize the spectral extremal graph. 
The previously known method seems too complicated to give a satisfactory answer 
even for small integers $k$; see, e.g., \cite{LL2025laa,LLLY2025}.   
 
 In what follows, we give a sharp refinement of the results \cite{LZS2024,LZZ2024} and determine the unique spectral extremal graphs for every large $m$.  Our result completely solves the problems in \cite{LZZ2024}.

\begin{theorem} \label{thm-theta-rpq}  
Let $1\le r\le p\le q$ be integers and $m$ be sufficiently large. Suppose that 
 $G$ is a $\theta_{r,p,q}$-free graph with $m$ edges and that one of the following holds: 
 \begin{itemize}
     \item[\rm (a)] 
    $\chi (\theta_{r,p,q})=3$ and 
   $r+p+q-1\in \{2k+1,2k+2\}$, where $k\ge 2$; 

  \item[\rm (b)] 
  $\chi (\theta_{r,p,q})=2$ and 
  $r+p+q-1\in \{2k+2,2k+3\}$, where $k\ge 1$. 
 \end{itemize} 
Then $\lambda(G) \le \lambda (S_{k,m})$,
where the equality holds if and only if $G=S_{k,m}$. 
\end{theorem}

\begin{proof}
For every $r\le p\le q$, at least one cycle 
of $\{C_{r+p},C_{r+q},C_{p+q}\}$ has an even length. 
We can color such an even cycle with two colors, and then color the remaining path of $\theta_{r,p,q}$ using at most one new color. Then $\theta_{r,p,q}$ is either bipartite or color-critical with $\chi (\theta_{r,p,q})=3$. 
In the case (a), we see that $\theta_{r,p,q}$ is a subgraph of $K_{k+1,k+1}^+$. In the case (b),  $\theta_{r,p,q}$ is a subgraph of $K_{k+1,k+2}$.  
{If $k=1$ in case (b), then $r+p+q\in \{5,6\}$, which together with $\chi (\theta_{r,p,q})=2$ implies $r=p=q=2$ and $\theta_{2,2,2}=K_{2,3}$. This case of $K_{2,t+1}$ was already studied by Zhai, Lin and Shu \cite{ZLS2021}. Thus, we get $\lambda (G)\le \sqrt{m} = \lambda (S_{1,m})$, with equality if and only if $G$ is a star.}  Now consider the case $k\ge 2$. Using Theorem \ref{thm-unify}, we get $\lambda (G)\le \lambda (S_{k,m})$.

{Next, we show that the above bound can be achieved. 
Write $s:=r+p+q$. For any
labeling $\{\ell_1,\ell_2,\ell_3\}=\{r,p,q\}$, the vertex
set of $\theta_{r,p,q}$ splits into the cycle of length
$\ell_1+\ell_2$ formed by the first two paths, and the
$\ell_3-1$ internal vertices of the third path. Hence
\[
\alpha(\theta_{r,p,q})\le\alpha(C_{\ell_1+\ell_2})
+\alpha(P_{\ell_3-1})
=\Bigl\lfloor\frac{\ell_1+\ell_2}{2}\Bigr\rfloor
+\Bigl\lceil\frac{\ell_3-1}{2}\Bigr\rceil.
\]
If $s$ is odd, then $\ell_1+\ell_2$ and $\ell_3$ have
different parities, and the right-hand side equals
$(s-1)/2$. If $s$ is even and
$\chi(\theta_{r,p,q})=3$, then $r,p,q$ do not all have the same
parity, so we may choose $\ell_1$ and $\ell_2$ of
different parities; then $\ell_1+\ell_2$ and $\ell_3$ are
odd, and the right-hand side equals $(s-2)/2$. If $s$ is
even and $\theta_{r,p,q}$ is bipartite, then $r,p,q$ are all even
and the right-hand side equals $s/2$.

Under condition (a), we have $s\in\{2k+2,2k+3\}$ and
$\chi(\theta_{r,p,q})=3$, so $\alpha(\theta_{r,p,q})\le k$ or $k+1$,
respectively. Under condition (b), we have
$s\in\{2k+3,2k+4\}$ and $\theta_{r,p,q}$ is bipartite, so
$\alpha(\theta_{r,p,q})\le k+1$ or $k+2$, respectively. In every
case, we have $\alpha(\theta_{r,p,q})\le s-k-2$. 

Finally, the vertices of $S_{k,m}$ outside the clique
$K_k$ form an independent set. Hence, a copy of $\theta_{r,p,q}$ in $S_{k,m}$ would yield an independent set of $\theta_{r,p,q}$ of size at least $v(\theta_{r,p,q})-k=s-1-k>\alpha(\theta_{r,p,q})$, a 
contradiction. Thus, $S_{k,m}$ contains no copy of $\theta_{r,p,q}$. This completes the proof. } 
\end{proof}

The case where $r=p=2$ and $q=3$ in Theorem \ref{thm-theta-rpq} refines a recent result of Gao and Li \cite{GL2025-Mar}. Moreover, in the case $r=1$, we obtain the following generalization of the result \cite{LZZ2024}. 

\begin{cor} \label{coro-theta}  
Let $2\le p\le q$ be integers with 
$p+q\in \{2k+1,2k+2\}$, where $k\ge 2$. 
If $m$ is sufficiently large and  
 $G$ is a $\theta_{1,p,q}$-free graph with $m$ edges, then 
\[ \lambda(G) \le \lambda (S_{k,m}),\] 
where the equality holds if and only if $G=S_{k,m}$. 
\end{cor}

The case $p=2$ in Corollary \ref{coro-theta} confirms two conjectures proposed by Liu and Li \cite{LL2025laa} and recovers the main result of Liu, Li, Li and Yu \cite{LLLY2025} by a quite different method.  

\subsubsection{Classical bipartite graphs} 

\label{sec-other-bipartite}

 F\"{u}redi \cite{Fur1996cpc} proved that 
if $s\ge t\ge 1$ and $G$ is an $n$-vertex $K_{s,t}$-free graph, then   
\begin{equation}\label{Fur96}
e(G) \le \frac{1}{2}{ (s-t+1)}^{{1}/{t}}n^{2- {1}/{t}} 
+ \frac{1}{2} tn^{2- {2}/{t}} + \frac{1}{2}tn. 
\end{equation}  
 Furthermore, Nikiforov \cite{Niki2010laa}  established a spectral generalization of (\ref{Fur96}) by showing that if $s\ge t\ge 3$ and $G$ is an $n$-vertex $K_{s,t}$-free graph, then 
\begin{equation} \label{Nik-Kst}
 \lambda (G)\le 
 (s-t+1)^{{1}/{t}} n^{1-{1}/{t}}+ 
 (t-2) n^{1-{2}/{t}}+(t-2).
\end{equation}
For bipartite graphs, determining the Tur\'{a}n number exactly, as well as its spectral counterpart,  has proved to be a significant and challenging problem; see \cite{AKS2003,FS13,CDT2023,CDT2024}.  

 For color-critical graphs $F$, the error term $o(1)$ in Theorems \ref{thm-ESS}, \ref{thm-spec-ESS} and \ref{thm:LLZ2025} can be removed. 
 For bipartite graphs $F$, the bounds in Theorems \ref{thm-ESS} and \ref{thm-spec-ESS} were improved to $O(n^{2-1/t})$ and $O(n^{1-1/t})$ for some $t\ge 1$, respectively, as seen in (\ref{Fur96}) and (\ref{Nik-Kst}). 
 In contrast, in the bipartite case, Theorem \ref{thm:LLZ2025} still maintains the order $m^{1/2}$ unless $F$ is a star; see Theorem \ref{thm-determine-M}.  

{It is instructive to compare this phenomenon with the vertex-spectral setting. Let $\mathrm{spex}(n,H)$ denote the maximum spectral radius of an $n$-vertex $H$-free graph. Xu, Gao and Chang \cite{XGC2024} proved a spectral version of the Erd\H{o}s--Simonovits theorem on finding an almost-regular subgraph, and deduced that for every
$\tfrac12<\xi\le 1$, one has $\mathrm{ex}(n,H)=O(n^{1+\xi})$ if and only if $\mathrm{spex}(n,H)=O(n^{\xi})$. In the vertex-spectral setting, forbidding a bipartite graph lowers the order of the maximum spectral radius. In contrast, Theorem~\ref{thm-determine-M} differs from this result and shows that in the edge-spectral setting, the main term is always $\sqrt{m}$ for every forbidden bipartite graph other than a star, and
the forbidden graph is visible only in the bounded
second-order term.}

In this section, we investigate the edge-spectral Tur\'{a}n problem for various bipartite graphs, including complete bipartite graphs, even cycles, matchings, paths, $1$-subdivisions of complete graphs and complete bipartite graphs, $d$-dimensional hypercubes, grids, $2\ell$-prisms and $2\ell$-cycles with all diagonals. 
We point out that the edge-spectral Tur\'{a}n problem has quite different behavior from the classical Tur\'{a}n problems for bipartite graphs. 
Indeed, as seen in (\ref{Fur96}) and (\ref{Nik-Kst}), the bounds are achieved by the projective norm graphs, which are pseudo-random graphs. In contrast, the extremal graph of the edge-spectral Tur\'{a}n problem is nearly a split graph $S_{k,m}$.

To state our result, we need to introduce some  notation. Let $\alpha (F)$ be the independence number of $F$. The following concept is frequently used in the study of graph Ramsey theory.

 \begin{defn*}
Let $F$ be a bipartite graph. 
The color surplus, denoted by $\sigma (F)$, is defined as the minimum size of a color class of $F$ over all possible proper $2$-colorings of the vertices of $F$. 
 \end{defn*}

 We obtain the following general theorem for bipartite graphs.

 \begin{theorem}  \label{thm-bipart}
     Suppose that $F$ is a bipartite graph with color surplus $ \sigma(F) = |F|- \alpha (F)$ and $F$ is not a star. 
     If $m$ is sufficiently large and $G$ is an $F$-free graph with  $m$ edges, then 
     \[ \lambda (G)\le \lambda (S_{\sigma(F) -1,m}), \] 
     where the equality holds if and only if $G=S_{\sigma (F) -1,m}$. 
 \end{theorem}
 
 \begin{proof}
We denote $\sigma:=\sigma(F)$ and $\alpha:=\alpha(F)$.
Since $F$ has no isolated vertices and is not a star, we
have $\sigma\ge 2$. By the definition of the color
surplus, $F$ is a subgraph of $K_{\sigma,\alpha}$, so
every $F$-free graph is $K_{\sigma,\alpha}$-free, and
hence also $K^{+}_{\sigma,\alpha}$-free. Note also that
$\sigma\le|F|/2\le\alpha$.

{Assume first that $\sigma\ge 3$. Applying
Theorem~\ref{thm-unify} with $k=\sigma-1$ and $t=\alpha-1$
gives $\lambda(G)\le\lambda(S_{\sigma-1,m})$, with
equality if and only if $G=S_{\sigma-1,m}$. 
Assume next that $\sigma=2$. Then $G$ is
$K_{2,\alpha}$-free. If $\alpha\ge 3$, the theorem of
Zhai, Lin and Shu \cite{ZLS2021} gives
$\lambda(G)\le\sqrt{m}=\lambda(S_{1,m})$ for sufficiently
large $m$, with equality if and only if $G$ is the star
$K_{1,m}=S_{1,m}$. If $\alpha=2$, then $G$ is $C_4$-free
since $K_{2,2}=C_4$, and the same conclusion follows from
the result of Nikiforov \cite{Niki2009}.}

It remains to show that $S_{\sigma-1,m}$ is $F$-free.
Suppose on the contrary that $S_{\sigma-1,m}$ contains a
copy of $F$. Then the vertices of $F$ are partitioned into two parts, one embedded into $K_{\sigma -1}$ and the other embedded into the independent set of $S_{\sigma -1,m}$. Hence $F$
would contain an independent set of size at least
$|F|-(\sigma-1)>\alpha$, which contradicts
$\sigma=|F|-\alpha$.
\end{proof}

\noindent 
{\bf Remark.}~In Example \ref{exam-chall}, we have  $\sigma (F)=n+1$ and $\alpha (F)=2(n+1-s)$, so $\sigma (F)\neq |F| -\alpha (F)$. Note that $S_{\sigma -1,m}$ is not the edge-spectral extremal graph for $F$. 
In conclusion, Example \ref{exam-chall} shows that the condition $\sigma (F)=|F| -\alpha (F)$ cannot be removed in Theorem \ref{thm-bipart}.

\medskip
As applications of Theorem \ref{thm-bipart}, we determine the edge-spectral extremal graph for some classical bipartite graphs. To begin with, we fix the following standard definitions. 

\begin{enumerate}
 
\item 
Let $M_{2k+2}$ be the matching with $2k+2$ vertices, and $P_k$ be the path on $k$ vertices. 

\item 
For a graph $F$, let $\mathrm{sub}(F)$ 
be the $1$-subdivision of $F$, which is obtained from $F$ by adding a new vertex to each edge of $F$. For instance, $\mathrm{sub}(K_3)=C_6$ and $\mathrm{sub}(K_{2,2}) = C_8$. 

\item 
Let $\Theta_{t, \ell}$ be the graph that consists of 
$t$ internally disjoint paths of length 
$\ell$ with the same endpoints. 
In particular, if $t=2$, then 
$\Theta_{2,\ell}= C_{2\ell}$; if $\ell =2$, then $\Theta_{t,2} =K_{2,t}$. 

\item 
The $d$-dimensional hypercube $Q_d$ is the graph whose vertex set is $\{0,1\}^d$ and in which two vertices are joined by an edge if they differ in exactly one coordinate. 

\item 
The grid $G_t$ is the graph with vertex set $[t]\times [t]$ in which two vertices are joined by an edge if they differ in exactly one coordinate and in that coordinate they differ by one. 

\item 
The $2\ell$-prism $C_{2\ell}^{\square}$ is the graph consisting of two vertex-disjoint $2\ell$-cycles and a matching joining the corresponding vertices on these two cycles. 
We see that $C_4^{\square}$ is the $3$-cube $Q_3$. 

\item 
Let $C_{2\ell}^{\mathrm{dia}}$ be the graph obtained from a $2\ell$-cycle by adding all diagonals, i.e., adding all chords joining vertices at maximum distance on the $2\ell$-cycle. 
For $\ell =3$, we see that $C_{6}^{\mathrm{dia}}$ is isomorphic to  $K_{3,3}$. 
Observe that $C_{2\ell}^{\mathrm{dia}}$ is bipartite if and only if $\ell$ is odd. 
\end{enumerate}

The classical Tur\'{a}n numbers of these bipartite graphs have recently been well-studied in the literature. 
Now, we summarize the solutions of the edge-spectral Tur\'{a}n problems in Table \ref{tab-SSSR}. 

\begin{table}[H]
\centering
\begin{tabular}{ccccccccccc}
\toprule
Forbidden substructures  & & Edge-spectral extremal graphs  \\
\midrule
$K_{k+1,t+1}$ with $1\le k\le t$  &  &  $S_{k,m}$ \\ 
$C_{2k+2}$ with $k\ge 1$  &  & $S_{k,m}$ \\  
$M_{2k+2}$ with $k\ge 2$  &  &  $S_{k,m}$ \\ 
$P_{2k+2},P_{2k+3}$ with $k\ge 2$  &  & $S_{k,m}$ \\ 
 $\mathrm{sub}(K_{k+1})$ with $k\ge 2$  & & $S_{k,m}$  \\
  $\mathrm{sub}(K_{s+1,t+1})$ with $s,t\ge 1$  & & $S_{s+t+1,m}$  \\
  $\Theta_{t,\ell}$ with $t\ge 2$ and odd $\ell \ge 3$ & & 
  $S_{\frac{1}{2}t(\ell -1),m}$  \\ 
   $\Theta_{t,\ell}$ with $t\ge 2$ and even $\ell \ge 2$ & & 
  $S_{\frac{1}{2}t(\ell -2)+1,m}$ \\ 
   $Q_d$ with $d\ge 2$  & & $S_{2^{d-1}-1,m}$ \\  
      $G_t$ with $t\ge 2$  & & $S_{\lfloor \frac{1}{2}t^2\rfloor-1,m}$ \\
      $C_{2\ell}^{\square}$ with $\ell\ge 2$  & & $S_{2\ell -1,m}$ \\ 
 $C_{2\ell}^{\mathrm{dia}}$ with odd $\ell\ge 3$  & & $S_{\ell -1,m}$ \\ 
\bottomrule 
\end{tabular}
\caption{Edge-spectral extremal graphs for some classical graphs}
 \label{tab-SSSR}
\end{table}

Setting $\ell =3$, 
we see that the theta graph $\theta_{1,3,3}$ is a subgraph of $C_{6}^{\mathrm{dia}}$, and $S_{2,m}$ is $\theta_{1,3,3}$-free. This also gives a solution to the unsolved case of a result of Li, Zhao and Zou \cite{LZZ2024}.

\subsubsection{Planar graphs and $k$-planar graphs}

A graph is called {\it planar} if it can be drawn in the plane without crossing edges. Bounding the spectral radius of planar graphs has a long history; see \cite{Hong1995,EZ2000} and references therein. 
A breakthrough of Tait and Tobin \cite{TT2017} states  that if $n$ is sufficiently large, then $K_2\vee P_{n-2}$ is the unique graph with maximum spectral radius over all $n$-vertex planar graphs. This makes significant progress towards the Boots--Royle--Cao--Vince Conjecture.

In this paper, we investigate the edge-spectral Tur\'{a}n problem for planar graphs. 
The edge-spectral extremal problem behaves quite
differently from the vertex-spectral problem. 
It is well-known that if $G$ is planar, 
then $G$ is $K_{3,3}$-free. Consequently, Theorem \ref{thm-bipart} implies $\lambda (G)\le \lambda (S_{2,m})$. This recovers the main result of Fan, Kang and Wu \cite{FKW2024} in a simple way.

A graph $G$ is called {\it $k$-planar} if it has a drawing in the plane $\mathbb{R}^2$ such that each edge of $G$ is crossed by at most $k$ other edges. 
 The vertex-spectral Tur\'{a}n problem for $1$-planar graphs was recently investigated by Zhang, Wang and Wang \cite{ZWW2024}. 
There are no related spectral extremal results for $k$-planar graphs with $k\ge 2$. 
In this paper, we provide the first such result for $k$-planar graphs. As a warm-up, 
we observe that every $1$-planar graph is $K_{3,7}$-free, since $K_{3,7}$ is not $1$-planar.  
We now extend this observation to $k$-planar graphs.

\begin{theorem} \label{thm-k-planar}
    For any fixed $k\ge 1$ and sufficiently large $m$, if $G$ is a $k$-planar graph with $m$ edges, then 
$ \lambda (G)\le \lambda (S_{2,m})$,  
   where the equality holds if and only if $G=S_{2,m}$.  
\end{theorem}

\begin{proof} 
Recall that the crossing number $\mathrm{cr}(G)$ of a graph $G$ is the minimum number of the total crossings of edges in a plane drawing of $G$. A well-known result of 
 Kleitman \cite{Kle1970} states that $\mathrm{cr}(K_{3,t})=\lfloor \frac{t}{2} \rfloor \cdot \lfloor \frac{t-1}{2}\rfloor$. Consequently, for any $k\ge 1$, there exists an integer $t_0=t_0(k)$ such that $K_{3,t}$ is not $k$-planar for every $t\ge t_0$. Then any $k$-planar graph $G$ is $K_{3,t}$-free for every $t\ge t_0$. Using Theorem \ref{thm-bipart}, we get $\lambda (G)\le \lambda (S_{2,m})$, and the equality holds if and only if $G=S_{2,m}$.  
 \end{proof}

\paragraph{Our approach.} 
Our proof of Theorem \ref{thm:color-critical-I} relies on 
the edge-spectral stability method (Theorem \ref{thm:nikiforov-stability}), which we expect to become a standard tool for the Brualdi--Hoffman--Tur\'{a}n problem.  
Let $G$ be an $m$-edge $F$-free graph with maximum spectral radius. Firstly, the stability method shows that $G$ is close in edit-distance to a regular complete $r$-partite graph $H=K_{U_1,\ldots ,U_r}$; see Claims \ref{claim-bound-m} and \ref{claim-subgraph-H}.  
The rest of the argument is carried out in Theorem \ref{thm:color-critical-when-stable}, and its starting point is that any possible imperfection in the partition of $H$ leads to a copy of the forbidden graph $F$. 
{Since $F$ is a subgraph of $K_r^{+}[t]$ for some $t$, it suffices to produce a copy of $K_r[t]$ together with one extra edge inside a partite set. The copy of $K_r[t]$ is supplied by a short probabilistic argument: since $G$ differs from $H$ in $o(m)$ edges, any sets $U_i'\subseteq U_i$
with $|U_i'|\ge \frac{1}{8r}|U_i|$ still span an almost complete $r$-partite graph, so sampling $t$ vertices from each $U_i'$ uniformly at random and applying the union bound yields a copy of $K_r[t]$ in $G[U_1',\dots,U_r']$; see Lemma~\ref{lem:find-turan-graph} and Claim~\ref{claim:4-2}.}
Building on this copy of $K_r[t]$, we show that no vertex of $G$ has many neighbors in every partite set of $H$; see Claim \ref{claim:4-3}. The second step is to find a large $r$-partite subgraph $G'$ of $G$ with $m'$ edges. To do this, we construct $r$ pairwise disjoint independent sets in $G$, say $V_1,\ldots ,V_r$; see Claim \ref{claim:4-4}. 
Let $S$ be the set of vertices outside $V_1,\ldots ,V_r$. 
Thirdly, we bound the coordinates of the Perron--Frobenius vector of $G$, and show that the vertices of $S$ contribute a small amount to the spectral radius of $G$. Let $\bm{x}$ be the unit Perron--Frobenius eigenvector of $G$, and denote 
$$A:= \Big(1-\frac{1}{r} \Big)2m',\quad 
B:= \Big(1- \frac{1}{r} \Big)2(m- m'), \quad 
X:=\sum_{v\in V(G')}x_v^2. $$
Then we prove that 
$$\sum_{uv\in E(G')} 2x_ux_v \le \sqrt{A} \cdot X$$
 and 
 $$\sum_{v\in S} x_v \sum_{w\in N(v)}\beta_w x_w \le \sqrt{(1-X)\cdot 2BX},$$
  where $\beta_w :=1$ if $w\in S$; and $\beta_w :=2$ if $w\notin S$. Consequently, by the Rayleigh quotient,  
  $$\lambda (G)\le \sqrt{A} X + \sqrt{B \cdot 2X(1-X)}\le \sqrt{A+B},$$
  where the second inequality applies the Cauchy--Schwarz inequality; 
   see Claim \ref{claim-bound-square} -- Claim \ref{upper-bound-2}.   
The above estimates can be regarded as a crucial step in our proof. {In this step, we use a stability result for the Perron--Frobenius eigenvector (Lemma~\ref{lem:Perron-Frobenius-stability}), which is a graph version of the Davis--Kahan perturbation bound.} Generally speaking, the Perron--Frobenius vector of $G$ is quite close to that of $T_{n,r}$ in the $\ell_2$-norm whenever $G$ has edit distance $o(m)$ from $T_{n,r}$. 

For Theorem \ref{thm:bipartite-rough}, let $F$ be an almost-bipartite graph and let $G$ be an $m$-edge $F$-free graph with maximum spectral radius. 
First of all, every graph of $\mathcal{M}_F$ is either empty or a non-empty graph with fewer than $|F|$ vertices. Moreover, for a graph $H$ and an integer $b> |F|$, the graph $H\vee bK_1$ is $F$-free if and only if $H$ is $\mathcal{A}_F$-free. 
Applying the edge-spectral stability in Theorem \ref{thm:nikiforov-stability}, we show that $G$ contains two disjoint vertex sets $A,B$ such that $d(G,K_{A,B})=o(m)$ and $G[A] \in \mathcal{M}_F$. Moreover, we show that $|A|=O(1)$ and $|B|=\Omega (m)$; see Claim \ref{cl:bound-A-B}.  
Using the stability for eigenvectors, we obtain  $x_v=\Omega (1)$ for every vertex $v\in A$; see Claim \ref{cl:weight-a}.  
We point out that $\{A,B\}$ is not necessarily a partition of vertices of $G$. For convenience, write $C:=V(G)\setminus A$. We show that $C$ is an independent set of $G$. Otherwise, we remove those edges inside $C$, and add the same number of pendant edges at a fixed vertex of $A$. This increases the spectral radius of $G$; see Claim \ref{lem:small-part-2}. Next, we show that $G[A]$ is non-empty and $|A| < |F|$; see Claim \ref{cl:A-less-F}.
Finally, we prove that all but at most one vertex of $C$ are complete to $A$. Otherwise, if $u,v\in C$ are not complete to $A$, where $d_u \ge d_v$, then moving an edge from $N(v)$ to $N(u)$ increases the spectral radius of $G$; see Claim \ref{cl:at-most-one}.

Theorem \ref{thm-determine-M} is a direct consequence of Theorem \ref{thm:bipartite-rough} together with Lemma~\ref{lem-estimate}. 
Indeed, let $A$ and $C$ be the vertex sets defined in Theorem \ref{thm:bipartite-rough}. 
With aid of the structural characterization in Theorem \ref{thm:bipartite-rough}, 
we see that $G$ differs from $K_{A,C}$ in $
{|A| \choose 2} + |A| =O(1)$ edges, which significantly improves the previous $o(m)$ bound with respect to $K_{A,B}$ obtained from the stability theorem.  
{Writing $G[A] \cong M \in \mathcal{M}_F$. 
Since $C$ is an independent set, and all but at most one vertex of $C$ are complete to $A$, this places $G$ between the joins $M\vee tK_1$ and $M\vee (t+1)K_1$ for some integer $t= \Theta(m)$. Thus, it suffices to determine the spectral radius of such a join $M\vee tK_1$. 
More precisely, we prove that $\lambda(M\vee tK_1)=\sqrt{m}+e(M)/v(M)+O(m^{-1/2})$, and the asymptotic formula for $G$ follows immediately; see Lemma~\ref{lem-estimate}.} 
This completes the proof of Theorem \ref{thm-determine-M}.

\paragraph{Organization.} 
In Section \ref{sec:prelim}, we provide some preliminaries that will be used in our proofs, including the edge-spectral stability theorem and the stability result for Perron--Frobenius eigenvectors in the $\ell_2$-norm.
In Section \ref{sec:proof-thm2-1}, 
we give the proof of Theorem \ref{thm:color-critical-I}. 
In Section \ref{sec:almost}, we present the proofs of Theorems \ref{thm:bipartite-rough} and \ref{thm-determine-M}, respectively.

\paragraph{Notation.} 
We write $G=(V,E)$ for a simple undirected graph with vertex set $V=\{v_1,\ldots ,v_n\}$ and edge set $E=\{e_1,\ldots ,e_m\}$, where $n=|V|$ and $m=|E|$. Sometimes, we write $|G|$ for the number of vertices of $G$. We write $ij$ for an edge if $i$ and $j$ are adjacent.  The degree of a vertex $v\in V$ is denoted by 
$d_v$, and the set of neighbors of $v$ is denoted by $N(v)$. For a subset $U \subseteq V(G)$, the set of neighbors of $v$ in $U$ is denoted by $N_U(v)$. 
The maximum degree of $G$ is denoted by 
$\Delta (G)$. 
Let $G[t]$ be the $t$-blow-up of $G$, which is a graph obtained from $G$ by replacing each vertex by an independent set of order $t$, and each edge by a copy of the complete bipartite graph $K_{t,t}$. 
We write $K_{V_1,\ldots ,V_r}$ for the complete $r$-partite graph on the vertex parts $V_1,\ldots ,V_r$. 
For disjoint sets $U,W\subseteq V$, 
we write $e(U)$ for the number of edges within $U$, and $e(U,W)$ for the number of edges between $U$ and $W$. 
For disjoint sets $U_1,\ldots ,U_r$, we write $G[U_1,\ldots ,U_r]$ for the $r$-partite subgraph of $G$ on the vertex sets $U_1,\ldots ,U_r$. 
The edit-distance between two graphs $G$ and $H$ is denoted by $d(G,H)$, which is the minimum number of edges we need to change (delete or add) to transform $G$ into $H$.  Let $G \vee H$ be the join graph consisting of $G$ and $H$ in which each vertex of $G$
is adjacent to each vertex of $H$.

Recall that the adjacency matrix $A_G$ 
of a graph $G$ is defined as a symmetric matrix of order $|V|\times |V|$ with $a_{i,j}=a_{j,i}=1$ if and only if $ij\in E$, and $a_{i,j}=a_{j,i}=0$ otherwise. 
All eigenvalues of $A_G$ are real and can be rearranged as $\lambda_1\ge \lambda_2\ge \cdots \ge \lambda_n$.  
The spectral radius $\lambda (G)$ is defined to be the maximum absolute value of the eigenvalues of $A_G$. By the Perron--Frobenius theorem, 
the spectral radius of a nonnegative matrix is actually the largest eigenvalue. Moreover, there exists a unit nonnegative eigenvector $\bm{x}\in \mathbb{R}^n$  corresponding to $\lambda (G)$. 
Such a vector is called the Perron--Frobenius eigenvector. For each vertex $v\in V$, 
we write $x_v$ for the coordinate of $\bm{x}$ corresponding to $v$. 
In the language of graphs, 
the eigen-equation $A_G\bm{x}= \lambda (G) \bm{x}$ can be read as $\lambda(G) x_v = \sum_{u\in N(v)} x_u$, and the Rayleigh quotient gives  
$\lambda (G) = 2 \sum_{uv\in E} x_ux_v$. 
Here and in the rest of the paper, we denote by $\sum_{uv \in E}$ the sum over each edge in $E$ \textbf{once}. When $G$ is a bipartite graph with given parts $A$ and $B$, we additionally assume that $u \in A$ and $v \in B$.

\section{Preliminary results}

\label{sec:prelim}

In this section, we collect a number of results that we will use throughout our proofs. 
Firstly, we need an edge-spectral stability result to obtain an approximate structure of the extremal graphs. 
For two disjoint vertex sets $A,B$, 
we write $K_{A,B}$ for the complete bipartite graph on the parts $A$ and $B$. 
For a vertex set $C$, we write $T_{C,r}$ 
for an $r$-partite Tur\'{a}n graph on $C$. 
The following stability result of Theorem \ref{thm:LLZ2025} was recently established by the authors \cite{LLZ2025-part-1}.

\begin{theorem}[Edge-spectral stability theorem \cite{LLZ2025-part-1}]
\label{thm:nikiforov-stability}
Let $F$ be a graph with $\chi (F)=r+1\ge 3$.  
For every $\varepsilon > 0 $, 
there exist $\delta>0$ and $m_0$ depending only on $F$ and $\varepsilon$ such that the following holds. If $G$ is an $F$-free graph with $m\ge m_0$ edges and $\lambda^2 (G)\ge (1- \frac{1}{r} - \delta) 2m$, then
\begin{itemize}
\item[\rm (a)] for $r = 2$, there exist disjoint vertex sets $A, B \subseteq V(G)$ such that $d(G, K_{A, B}) \leq \varepsilon m$.

\item[\rm (b)]
 for $r \geq 3$, there exist a vertex set $C \subseteq V(G)$ and an $r$-partite Tur\'{a}n graph $T_{C, r}$ on $C$ such that $d(G, T_{C, r}) \leq \varepsilon m$. 
 \end{itemize}
\end{theorem}
 
By Theorem \ref{thm:nikiforov-stability}, we obtain the following corollary for forbidding bipartite graphs. 

\begin{cor} \label{cor-bi-stability}
    Let $F$ be a bipartite graph. 
    For every $\varepsilon >0$, there exist $\delta >0$ and $m_0$ depending only on $F$ and $\varepsilon$ such that the following holds.  
    If $G$ is an $F$-free graph with $m\ge m_0$ edges and $\lambda (G)\ge (1- \delta)\sqrt{m}$, then there exist disjoint vertex sets $A,B \subseteq V(G)$ such that $d(G,K_{A,B})\le \varepsilon m$. 
\end{cor}

\begin{proof}
    Since $F$ is bipartite, there exist positive integers $s,t$ such that $F$ is a subgraph of $ K_{s,t}^+$. We see that $G$ is $K_{s,t}^+$-free and $\chi (K_{s,t}^+)=3$. Applying Theorem \ref{thm:nikiforov-stability} (a), the desired result holds.   
\end{proof}

We mention that the edge-spectral stability above is stronger than the classical Erd\H{o}s--Simonovits stability theorem \cite{Sim1966} and its vertex-spectral version due to Nikiforov \cite{Niki2009jgt}. 
{Such an edge-spectral stability result is interesting in its own right, and it would serve as a stepping stone in proving exact results
when we forbid a specific substructure.}  

The following lemma shows that if two graphs are close in the edit distance, then their Perron--Frobenius eigenvectors are close in the $\ell_2$-norm. 
{This is a graph version of the Davis--Kahan perturbation bound in matrix theory; we include a short proof for completeness.}  

\begin{lemma}[Stability of Perron--Frobenius eigenvector]
    \label{lem:Perron-Frobenius-stability}
    Let $G$ and $H$ be two graphs on the same vertex set $V$ such that $d(G,H) \leq e$. Let $\bm{x}$ and $ \bm{y}$ be the unit Perron--Frobenius eigenvectors of $G$ and $H$, respectively. Then 
    \begin{equation} \label{lem-3-4-1}
        \sum_{v \in V} (x_v - y_v)^2 \leq \frac{8\sqrt{e}}{\lambda_1(H) - \lambda_2(H)}
    \end{equation}
    and
    \begin{equation} \label{lem-3-4-2}
        \sum_{v \in V} |x_v^2 - y_v^2| \leq \frac{8 e^{1/4}}{\sqrt{\lambda_1(H) - \lambda_2(H)}}.
    \end{equation}
\end{lemma}

\noindent 
\textbf{Remark.}~
In this paper, we always set $H$ to be a complete multipartite graph that is obtained from Theorem \ref{thm:nikiforov-stability}. In this case, it follows that $\lambda_2(H) = 0$ and $\lambda_1(H) \geq \sqrt{e(H)}$.

\begin{proof}
As $\bm{x}$ realizes the spectral radius of $A_G$, we have
$\bm{x}^T A_G \bm{x} \ge \bm{y}^T A_G \bm{y}$. Note that 
\[ \bm{x}^TA_G \bm{x} = \sum_{ij\in E(G)} 2x_ix_j = \sum_{ij\in E(H)} 2x_ix_j - 
\sum_{ij\in E(H)\setminus E(G)} 2x_ix_j 
+ \sum_{ij\in E(G)\setminus E(H)} 2x_ix_j. \]
Since $E(G) \setminus E(H)$ has at most $e$ edges, 
 the third term in the above is at most $\sqrt{2e}$, 
 and hence  
$\bm{x}^T A_G \bm{x} \le \bm{x}^T A_H \bm{x} + \sqrt{2e}$. Similarly, we have $\bm{y}^T A_H \bm{y} \le \bm{y}^T A_G \bm{y} + \sqrt{2e}$. Then 
\[ \bm{x}^TA_H \bm{x} + \sqrt{2e} \ge 
   \bm{x}^T A_G \bm{x} \ge \bm{y}^T A_G \bm{y} \ge 
\bm{y}^T A_H \bm{y} - \sqrt{2e}. \]
Consequently, we get 
\begin{equation} \label{eq-H-ge}
\bm{x}^T A_H \bm{x} \geq \bm{y}^T A_H \bm{y} - 2\sqrt{2e} \ge \lambda_1(H) - 4\sqrt{e}.
\end{equation}
Let $\bm{y}_1,\bm{y}_2,\ldots,\bm{y}_n$ be the orthogonal unit eigenvectors of $A_H$ corresponding to $\lambda_1, \lambda_2, \ldots, \lambda_n$, where we write  $\bm{y}_1=\bm{y}$. Then 
$\lVert \bm{x} \rVert_2^2 =\sum_{i=1}^n \langle \bm{x}, \bm{y}_i \rangle^2=1$. 
Using the spectral decomposition of $A_H$, 
we have  $A_H=\sum_{i=1}^n \lambda_i\cdot\bm{y}_i\bm{y}_i^T$.  
Thus, it follows that 
\begin{align}   
\bm{x}^T A_H \bm{x} =\sum_{i=1}^n \lambda_i \cdot \langle \bm{x}, \bm{y}_i \rangle^2 
\leq \lambda_1 \cdot \langle \bm{x}, \bm{y} \rangle^2 + \lambda_2 \cdot (1 - \langle \bm{x}, \bm{y} \rangle^2). \label{eq-H-le}
\end{align}
Combining (\ref{eq-H-ge}) with (\ref{eq-H-le}), we conclude that
$$1 - \langle \bm{x}, \bm{y} \rangle^2 \leq \frac{4\sqrt{e}}{\lambda_1 - \lambda_2}.$$
Finally, note that $\langle \bm{x}, \bm{y} \rangle \in [0, 1]$. Thus, we obtain the first inequality \eqref{lem-3-4-1} by
$$\sum_{v \in V} (x_v - y_v)^2 = \langle \bm{x}-\bm{y}, \bm{x}-\bm{y}\rangle=2 (1 - \langle \bm{x}, \bm{y} \rangle) \leq 2(1 - \langle \bm{x}, \bm{y} \rangle^2).$$
To show (\ref{lem-3-4-2}), we apply 
the Cauchy--Schwarz inequality
$$\sum_{v \in V} |x_v^2 - y_v^2| \leq \left(\sum_{v \in V} (x_v - y_v)^2\right)^{1/2} \left(\sum_{v \in V} (x_v + y_v)^2\right)^{1/2} \leq 2\left(\sum_{v \in V} (x_v - y_v)^2\right)^{1/2}$$
as $(x_v+y_v)^2\le 2(x_v^2+y_v^2)$. The conclusion then follows from the first inequality (\ref{lem-3-4-1}).
\end{proof}

Recall that $K_r[t]$ is the $t$-blow-up of $K_r$, which is the complete $r$-partite graph with each part having $t$ vertices. We write $K_{V_1,\ldots ,V_r}$ for the complete $r$-partite graph with vertex parts $V_1,\ldots ,V_r$. 
We need the following lemma to find a large copy of $K_{r}[t]$ in a dense $r$-partite graph.

\begin{lemma}
    \label{lem:find-turan-graph}
    For every $t$ and $r \geq 2$, there is a constant $\delta_{t, r} >0$ such that the following holds. Suppose that $G$ is an $r$-partite graph with partite sets $V_1, \dots, V_r$ such that $\abs{V_i} \geq t$ for each $i$. If 
    $$e(K_{V_1, \dots, V_r}) - e(G) \leq \delta_{t, r} \cdot  \min_{i \neq j} \abs{V_i} \abs{V_j},$$
    then $G$ contains a copy of the complete $r$-partite graph  $K_r[t]$.
\end{lemma}

The proof of Lemma \ref{lem:find-turan-graph} proceeds using the standard probabilistic method. 

\begin{proof}
    We take $\delta_{t, r} :=t^{-2}r^{-2}$ and sample a $t$-element subset $V_i' := \{v_{i, 1}, \dots, v_{i, t}\}$ from each $V_i$ uniformly at random. For two sampled vertices $v_{i, a}$ and $v_{j, b}$ from different parts $V_i$ and $V_j$, respectively, the probability that there is no edge between $v_{i,a}$ and $v_{j,b}$ is at most
    $$\frac{e(K_{V_1, \dots, V_r}) - e(G)}{\abs{V_i} \abs{V_j}} \leq \delta_{t, r}.$$
    By the union bound, the probability that $G[V_1', \dots, V_r']$ is not isomorphic to $K_r[t]$ is at most
    $\binom{r}{2} \cdot t^2 \cdot \delta_{t, r} < \frac{1}{2}
     < 1$. 
 Hence, there exists a sampling of $V_i'$ from each $V_i$ such that $G[V_1',\ldots ,V_r']$ forms a copy of $K_r[t]$. So $G$ contains a copy of $K_r[t]$, as desired. 
\end{proof}

The following lemma is folklore. 

\begin{lemma} \label{lem:e-T}
    Let $T_{n,r}$ be the $r$-partite Tur\'{a}n graph on $n$ vertices. Then 
    \[  \left(1- \frac{1}{r} \right)\frac{n^2}{2} - \frac{r}{8} \le 
    e(T_{n,r})\le \left(1- \frac{1}{r} \right)\frac{n^2}{2} .  \]
\end{lemma}

\begin{proof}
    We denote $t:=\lfloor n/r\rfloor$ and write 
    $n=rt +s$ for some integer $0\le s <r$. Then $T_{n,r}$ has $s$ partite sets of size $t+1$, and $r-s$ partite sets of size $t$. It follows that 
    \[ e(T_{n,r}) = {n \choose 2} - s{t+1 \choose 2} - (r-s){t \choose 2}= \left(1- \frac{1}{r} \right)\frac{n^2}{2} - 
    \frac{s(r-s)}{2r}.\]
    By the AM-GM inequality, we have $s(r-s) \le r^2/4$. Thus, the desired lower bound holds.   
\end{proof}

\section{Color-critical graphs}
\label{sec:proof-thm2-1}

Recall that the case $F=K_{r+1}$ in Theorem \ref{thm:color-critical-I} was proved by Nikiforov \cite{Niki2002}, whose proof is based on the Motzkin--Straus inequality. 
This inequality states that if $G$ is an $n$-vertex $K_{r+1}$-free graph, and $x_1,x_2,\ldots,x_n$ are non-negative weights of vertices of $G$ satisfying $\sum_{i=1}^n x_i =1$, 
then $\sum_{ij\in E} x_ix_j \le \frac{1}{2}(1- \frac{1}{r})$. 
However, his proof does not seem to apply to color-critical graphs, as the Motzkin--Straus inequality cannot be extended to $F$-free graphs in general.

\subsection{Forcing optimal partitions via approximate structures}

Our proof of Theorem \ref{thm:color-critical-I} proceeds via the edge-spectral stability in Theorem \ref{thm:nikiforov-stability} together with a couple of observations when $G$ is close to a complete $r$-partite graph. 
Let $K_r^+[t]$ denote the graph formed by adding an edge to a partite set of the complete $r$-partite graph $K_{r}[t]$. Clearly, if a graph $F$ is a color-critical graph with $\chi (F)=r+1$, then $F$ is a subgraph of $K^+_{r}[t]$ for some integer $t$. Thus, it suffices to establish Theorem \ref{thm:color-critical-I} for $F = K^+_{r}[t]$.

\begin{theorem}
\label{thm:color-critical-when-stable}
For every $t\ge 2$ and $r \geq 2$, there are $\varepsilon = \varepsilon_{t, r} > 0$ and $n_0 = n_{t, r}$ such that the following holds. Suppose that $G$ is a $K_{r}^+[t]$-free graph with $m$ edges, and $H$ is a complete $r$-partite graph where 
$V(H) \subseteq V(G)$ and each partite set of $H$ has at least $n_{0}$ vertices, and $H$ differs from $G$ in at most $\varepsilon m$ edges. Moreover, suppose that $H$ is regular when $r \geq 3$. Then 
$$\lambda^2(G) \leq \left( 1- \frac{1}{r}\right) \cdot 2m,$$
where equality holds if and only if $G$ is a complete bipartite graph when $r = 2$, or $G$ is a regular complete $r$-partite graph when $r \geq 3$. 
\end{theorem}

\noindent 
\textbf{Remark.}~The statements for the cases $r = 2$ and $r \geq 3$ are slightly different, but the proofs are identical. Theorem \ref{thm:color-critical-when-stable} will be used in our proofs of both Theorems \ref{thm:color-critical-I} and \ref{thm:bipartite-rough}.

\begin{proof}
   In our argument, we take $\varepsilon_{t, r} = (40r)^{-8}\delta_{t, r}$ and $n_{t, r} = 8rt$, where $\delta_{t,r}=t^{-2}r^{-2}$ is determined in Lemma \ref{lem:find-turan-graph}. 
    By the assumption, $H$ is a complete $r$-partite graph. Let $U_1,U_2, \dots, U_r$ be the partite sets of $H$. 
     In the sequel, we prove a sequence of claims.

    \begin{claim} \label{claim:4-2}
        If $U_i'$ is a subset of $U_i$ with $|U_i'|\ge \frac{1}{8r}|U_i|$ for every $i\in [r]$, then there is a copy of the complete $r$-partite graph $K_r[t]$ in $G[U_1', \dots, U_r']$.
    \end{claim}

    \begin{poc}
        We have $|{U_i'}| \geq \frac{1}{8r}|U_i| \ge 
        \frac{n_{t,r}}{8r} \geq t$. Since $d(G, H) \leq \varepsilon m$, we have 
        \[ e(K_{U_1', \dots, U_r'}) - e(G[U_1', \dots, U_r'] )\leq \varepsilon m. \] 
        Thus, by assumption for any $i \neq j$, we have
    \begin{align*}
    |U_i'| |{U_j'}| &\ge \frac{1}{64r^2}|U_i| |U_j|
    = \frac{1}{64r^2\binom{r}{2}} e(K_{U_1, \dots, U_r}) \geq \frac{1}{64r^2 \binom{r}{2}} \cdot (1 - \varepsilon) m \\
    &\geq \frac{1}{64r^4 \varepsilon} (e(K_{U_1', \dots, U_r'}) - e(G[U_1', \dots, U_r'] )).
    \end{align*}
    Applying Lemma \ref{lem:find-turan-graph}, we conclude that $G[U_1', \dots, U_r']$ contains a copy of $K_{r}[t]$.
     \end{poc}

    \begin{claim} \label{claim:4-3}
        For any vertex $v\in V(G)$, there exists some $i\in [r]$ such that $|N_{U_i}(v)| \leq \frac{1}{8r} |{U_i}|$. 
    \end{claim}

    \begin{poc}
    Suppose on the contrary that 
    there exists a vertex $v\in V(G)$ such that $|N_{U_i}(v)|> \frac{1}{8r}|U_i| $ for every $i\in [r]$. 
    We denote by $U_i'=N_{U_i}(v)$ for every $i\in [r]$. 
    Then Claim \ref{claim:4-2} shows that the $r$-partite subgraph $G[U_1', \dots, U_r']$ contains a copy of $K_{r}[t]$. This copy, together with the vertex $v$, produces a copy of $K_{r}^+[t]$, 
    which is a contradiction. 
    \end{poc}

      For each $v \in V(G)$, let $N_{U_i}(v)$ denote the set of vertices of $U_i$ that are adjacent to $v$ in the graph $G$. In other words, $N_{U_i}(v)=N_G(v)\cap U_i$ for every $i\in [r]$.   
     The key idea is to define the following vertex sets. 
     For each $i \in [r]$, let $V_i$ be the set of vertices defined as 
    \[  V_i:=\left\{v\in V(G) : |N_{U_j}(v)| \ge \frac{2}{3}|U_j| \text{~for any $j\in [r]\setminus \{i\}$} \right\} . \] 

    \begin{claim}  \label{claim:4-4}
    The sets $V_1,V_2,\ldots ,V_r$ are pairwise disjoint independent sets of $G$.
    \end{claim}

    \begin{poc}
    We assume on the contrary that there exists a vertex $v\in V_i\cap V_j$ for some $i\neq j$. By the definition of $V_i$, we know that $|N_{U_k}(v) | \geq \frac{2}{3} \abs{U_k}$ for any $k \neq i$. 
    Similarly, the definition of $V_j$ implies that $|N_{U_k}(v)|\ge \frac{2}{3}|U_k|$ for any $k\neq j$. 
Combining these, we get $|N_{U_k}(v)| \ge \frac{1}{8r}|U_k|$ for every $k\in [r]$,  
    which  contradicts Claim \ref{claim:4-3}. 
    Thus $V_1,V_2,\ldots ,V_r$ are pairwise disjoint. 

Next, we show that each $V_i$ is independent. 
   Suppose on the contrary that there exists an edge $uu'$ in $G[V_i]$ for some $i\in [r]$. We denote $U_i':=U_i\setminus \{u,u'\}$. 
    For each $j\in [r]\setminus \{i\}$, we denote $U_j':=N(u)\cap N(u')\cap U_j$. 
    Since $|U_i|\ge n_{t,r}\ge 4$, we have $|{U_i'}| \ge |U_i|-2 \geq \frac{1}{2} |{U_i}|$. 
    For each $j \in [r]\setminus \{i\}$, the definition of
$V_i$ gives $|N_{U_j}(u)|\ge \frac{2}{3}|U_j|$ and
$|N_{U_j}(u')|\ge \frac{2}{3}|U_j|$, hence
    $$|{U_j'}| \geq |N_{U_j}(u)| + |N_{U_j}(u')| - |{U_j}| 
    \geq \frac{1}{3} |{U_j}|.$$
    By Claim \ref{claim:4-2}, there is a copy of $K_{r}[t]$ in $G[U_1', \dots, U_r']$. Combining with the edge $uu'$ in $G[V_i]$, we find a copy of $K_{r}^+[t]$ in $G$, which is a contradiction. So $V_i$ is an independent set of $G$. 
    \end{poc}

    We denote by $\bm{x} \in \mathbb{R}^{|V|}$ the unit 
    Perron--Frobenius eigenvector of $G$.

    \begin{claim}
        \label{claim-bound-square}
        Let $R$ be the set of vertices that are not in any of the $U_i$'s. Then 
        \[ \sum_{w \in R} x_w^2 + \sum_{i = 1}^r \sum_{w \in U_i} \left| x_w^2 - \frac{1}{r|U_i|} \right| \leq 9 \varepsilon^{1/4} \]
    \end{claim}

    \begin{poc} 
    Recall that $H=K_{U_1,\ldots ,U_r}$. 
     In what follows, we apply Lemma \ref{lem:Perron-Frobenius-stability} to $G$ and $H$.  
    The graph $H$ can be viewed as a graph on the vertex set of $G$ by adding isolated vertices. 
    Then $V(G)=V(H)\cup R$. 
    The unit Perron--Frobenius eigenvector $\bm{y}$ of $H$ is given as  
    $$\bm{y}_v = \begin{cases}
        \frac{1}{\sqrt{r \abs{U_i}}}, & \text{if}~v \in U_i; \\
        0, & \text{if}~v \in R.
    \end{cases}$$ 
   Note that $H$ is a complete $r$-partite graph 
   which is regular when $r\ge 3$, so that $\lambda_1(H)^2 =
   (1-\frac{1}{r})2e(H)$; for $r=2$, this is the identity
   $\lambda_1(K_{a,b})=\sqrt{ab}=\sqrt{e(H)}$.
    Together with $e(H)\ge (1- \varepsilon)m$, this gives  
   \[ \lambda_1(H)=\sqrt{\Big(1- \frac{1}{r}\Big)2e(H)}\ge \sqrt{0.9m} \quad \text{and} \quad  \lambda_2(H)=0. \]  
   By Lemma \ref{lem:Perron-Frobenius-stability}, we get 
    $$\sum_{w \in R} x_w^2 + \sum_{i = 1}^r \sum_{w \in U_i} \abs{x_w^2 - y_w^2} \leq \frac{8(\varepsilon m)^{1/4}}{\sqrt{\lambda_1(H)}}\le 9 \varepsilon^{1/4}.$$
    Recalling the expression of the vector $\bm{y}$, we have
    \begin{equation*} 
    \sum_{w \in R} x_w^2 + \sum_{i = 1}^r \sum_{w \in U_i} \left| x_w^2 - \frac{1}{r|U_i|} \right| \leq 9 \varepsilon^{1/4},
    \end{equation*} 
    as desired. 
    \end{poc}

    \begin{claim} \label{claim:4-5}
Let $S$ be the set of vertices not in any of the $V_i$'s.  
 Then for $v\in S$, 
    $$\sum_{w \in N(v)} x_w^2 \leq 1 - \frac{4}{3r} + \frac{1}{4r^2}.$$
    \end{claim}

    \begin{poc}   
    By Claim \ref{claim:4-3}, for the vertex $v\in V(G)$, 
     there exists an index $i_0\in [r]$ such that $|N_{U_{i_0}}(v)| \leq \frac{1}{8r} |{U_{i_0}}|$. On the other hand, we see that $v\notin V_{i_0}$ since $v\in S$. 
     By the definition of $V_{i_0}$, there exists an index $j_0\neq i_0$ such that $|N_{U_{j_0}}(v)|< \frac{2}{3}|U_{j_0}|$. 
    Note that 
    \[ \sum_{w\in N(v)}  x_w^2 \le 
    \sum_{w\in N_R(v)} x_w^2 + 
    \sum_{w\in N_{\bar{R}}(v)} x_w^2. \]
    By Claim \ref{claim-bound-square}, we get 
    \[   \sum_{w\in N_R (v)} x_w^2 \le \sum_{w\in R} x_w^2 \le 9 \varepsilon^{1/4}.  \]
Using Claim \ref{claim-bound-square} again, we have
\[ \sum_{i=1}^r \sum_{w\in N_{U_i}(v)} \left( x_w^2 - \frac{1}{r|U_i|}\right) \le 9 \varepsilon^{1/4}. \]
Since $\overline{R}=U_1\sqcup \cdots \sqcup U_r$, it follows that   
\[  \sum_{w\in N_{\overline{R}}(v)} x_w^2 = \sum_{i=1}^r \sum_{w\in N_{U_i}(v)} x_w^2 \le \sum_{i=1}^r \frac{|N_{U_i}(v)|}{r|U_i|} 
+ 9\varepsilon^{1/4}. \]
Therefore, we conclude that 
    \begin{align*}
    \sum_{w \in N(v)} x_w^2 
    &\leq \sum_{i = 1}^r \frac{|N_{U_i}(v)|}{r |U_i|} + 18\varepsilon^{1/4} \\
    &\leq \frac{|N_{U_{i_0}}(v)|}{r|U_{i_0}|} + \frac{|N_{U_{j_0}}(v)|}{r|U_{j_0}|} + \frac{r-2}{r} + 18\varepsilon^{1/4} \\
    &\leq \frac{1}{8r^2} + \frac{2}{3r} + \frac{r - 2}{r} + \frac{1}{8r^2} + 18\varepsilon^{1/4} \\
    & \le 1 - \frac{4}{3r} + \frac{1}{2r^2},
    \end{align*}
    as desired. 
    \end{poc}

    \begin{claim}   \label{claim:4-6}
    We have $\sum\limits_{v \in \overline{S}} x_v^2\geq 1 - \frac{1}{16r^2}$. 
    \end{claim}

    \begin{poc}
    The observation is that the majority of each $U_i$ lies in $V_i$. For every vertex $w\in U_i \setminus V_i$, there is an index $j=j(w)\in [r]\setminus \{i\}$ such that $w$ has fewer than $\frac{2}{3} |U_j|$ neighbors in $U_j$, and hence $w$ misses at least $\frac{1}{3}|U_j|$ neighbors in $U_j$. 
   Note that for all vertices $w\in U_i \setminus V_i$, the index $j(w)$ may be distinct, but $|U_{j(w)}|$ has the same size when $r\ge 3$ as $H$ is regular; and that $j(w)$ is the unique index other than $i$ when $r=2$. 
    To transform $G$ to $H=K_{U_1,\ldots ,U_r}$, we must add these missing edges on $w$. 
    Since $G$ differs from $H$ in at most $\varepsilon m$ edges,  we get 
    $$|{U_i \setminus V_i}| \cdot  \frac{1}{3}  |U_j| \leq \varepsilon m \le \frac{\varepsilon}{1-\varepsilon}e(H) \le  \frac{\varepsilon}{1-\varepsilon} {r \choose 2} |U_i| \cdot  |U_j| 
    \leq \varepsilon r^2 |{U_i}| \cdot  |{U_j}|,$$
    where the second inequality holds by $e(H)\ge (1- \varepsilon) m$, and the third holds because $H$ is regular when $r\ge 3$. 
    Hence, for every $i\in [r]$, we have  
    $|{U_i \setminus V_i}| \leq 3\varepsilon r^2 |{U_i}|$, 
    which gives 
    \[ |U_i\cap V_i| = |U_i| - |U_i\setminus V_i| \ge (1- 3\varepsilon r^2)|U_i|. \]
    Applying Claim \ref{claim-bound-square}, we obtain
    \begin{align*}
    \sum_{v \in \overline{S}} x_v^2 &\geq \sum_{i = 1}^r \sum_{v \in U_i \cap V_i} x_v^2 
    \geq \sum_{i = 1}^r \frac{|{U_i \cap V_i}|}{r |{U_i}|} - 9\varepsilon^{1/4} 
    \geq (1 - 3\varepsilon r^2) - 9\varepsilon^{1/4} \geq 1 - \frac{1}{16r^2},
    \end{align*}
    as needed. 
    \end{poc}

    We now have all the ingredients to bound the spectral radius of $G$. 
    Let $G':=G[V_1,\ldots ,V_r]$ be the induced subgraph of $G$ with vertex sets $V_1 \sqcup \ldots  \sqcup V_r$. Let $m'$ be the number of edges in $G'$. 
    By Claim \ref{claim:4-4}, we know that there is no edge of $G$ within $V_i$ for every $i\in [r]$. So all edges of $G$ are contained in $G'$ or incident to vertices of $S$. 
    Then 
    \begin{equation} \label{upper-two-terms}
        \lambda(G) = \bm{x}^T A_G \bm{x} \leq \sum_{u v \in E(G')} 2 x_u x_v + 
    \sum_{v \in S} \sum_{w \in N(v)}  x_v\cdot  \beta_w x_w, 
    \end{equation}
   where the function $\beta : V(G) \to \mathbb{R}$ is defined as  
    \[  \beta_w:= \begin{cases}
        1, & \text{if $w\in S$;} \\
        2, & \text{if $w\notin S$.}
    \end{cases}   \] 
   We define the following shorthand notation for simplicity: 
    $$A := \left(1- \frac{1}{r}\right) \cdot 2m', \quad 
    B :=\left(1- \frac{1}{r}\right) \cdot 2(m -m'), 
    \quad 
    X := \sum_{v \in \bar{S}} x_v^2.$$
To bound the terms on the right-hand side of (\ref{upper-two-terms}), we prove the following two claims.

    \begin{claim} \label{upper-bound-1}
        We have $\sum\limits_{u v \in E(G')} 2 x_u x_v \leq \sqrt{A}X$. 
    \end{claim}

    \begin{poc}
    By Claim \ref{claim:4-4}, we know that $G'$ is an $r$-partite  subgraph with $m'$ edges.  Applying Nikiforov's bound (\ref{eq-Niki-2002-cpc}) to the $K_{r+1}$-free subgraph $G'$, we get that $\lambda (G')\le \sqrt{A}$. 
    Recall that $V(G')=V_1\cup \cdots \cup V_r=\overline{S}$. 
    The Rayleigh formula implies 
    $$\sum_{u v \in E(G')} 2 x_u x_v \leq \lambda (G')\cdot X \le  \sqrt{A}X.$$
    \end{poc}

\begin{claim} \label{upper-bound-2}
    We have $\sum\limits_{v \in S} x_v 
    \sum\limits_{w \in N(v)} \beta_w x_w \le \sqrt{(1-X)\cdot 2BX}$.
\end{claim}

    \begin{poc}
    By the Cauchy--Schwarz inequality, we obtain 
    \begin{align*}
    \sum_{v \in S} x_v \sum_{w \in N(v)} \beta_w x_w &\leq \left(\sum_{v \in S} x_v^2\right)^{\!\!1/2} \left(\sum_{v \in S} \left(\sum_{w \in N(v)} \beta_w x_w\right)^{\!\!2}\right)^{\!\!1/2}  \!\!\!\! \\
    &= (1 - X)^{1/2} \left(\sum_{v \in S} \left(\sum_{w \in N(v)} \beta_w x_w\right)^{\!\!2}\right)^{\!\!1/2}.
    \end{align*}
    To bound the right-hand side, we define the function $ \gamma: V(G)\to \mathbb{R}_+$ as below
     \[ \gamma_w:= \begin{cases}
        2, & \text{if $w\in S$;} \\
        1, & \text{if $w\notin S$.}
    \end{cases}   \] 
    By the Cauchy--Schwarz inequality again, for each $v \in S$, we have
    $$\left(\sum_{w \in N(v)} \beta_w x_w\right)^2 \leq 
    \left(\sum_{w \in N (v)} \beta_w^2 \gamma_w x_w^2\right) \left(\sum_{w \in N(v)}\frac{1}{\gamma_w}\right).$$
   Note that $\beta_w^2 \gamma_w \le 4$ for every $w\in V(G)$. By Claims \ref{claim:4-5} and \ref{claim:4-6}, 
    it follows that 
    $$\sum_{w \in N(v)} \beta_w^2\gamma_w x_w^2  
    \leq 4\left(1 - \frac{4}{3r} + \frac{1}{2r^2}\right) \le 4 \left(1 - \frac{1}{r}\right) 
    \left(1 - \frac{1}{16r^2}\right) \le 4 \left(1 - \frac{1}{r}\right) X.$$
    Now observe the identity
    $$\sum_{v \in S}\sum_{w \in N(v)}\frac{1}{\gamma_w} = 
    e(S,\bar{S}) + e(S) = m - m'.$$
   Thus, we conclude that 
    $$\sum_{v \in S} \left(\sum_{w \in N(v)} \beta_w x_w\right)^2 \leq 4 \left(1 - \frac{1}{r}\right) X \cdot (m - m') = 2BX.$$ 
    \end{poc}

    Substituting Claims \ref{upper-bound-1} and \ref{upper-bound-2} back to (\ref{upper-two-terms}), we obtain
    $$\lambda(G) \leq \sqrt{A}X + \sqrt{(1 - X)\cdot 2BX}.$$
    We use the Cauchy--Schwarz inequality one last time to get
    \begin{align*}
    \sqrt{A}X + \sqrt{B \cdot 2X(1 - X)} \leq 
    \sqrt{A + B} \sqrt{X^2 + 2X(1 - X)} \leq \sqrt{A + B} = \sqrt{\Big( 1-\frac{1}{r}\Big)2m}, 
    \end{align*}
    which is the result we want. 
    
Moreover, suppose that $\lambda (G)=\sqrt{(1-\nicefrac{1}{r})2m}$. Then equality must hold in the last step above, which forces $X^2+2X(1-X)=1$, that is, $X = 1$. 
Substituting $X=1$ into Claims
\ref{upper-bound-1} and \ref{upper-bound-2} and using
(\ref{upper-two-terms}), we get $\lambda (G)\le \sqrt{A}
=\sqrt{(1-\nicefrac{1}{r})2m'}$. Since $m'\le m$, this
forces $m'=m$, and hence every edge of $G$ lies in $G'$.
In other words, every vertex of $S$ is isolated in $G$,
so that $G = G[V_1, \dots, V_r]$ up to isolated vertices.
By Claim \ref{claim:4-4}, this implies that $G$ is an
$r$-partite graph, and in particular $G$ is
$K_{r+1}$-free. Since $\lambda(G) =
\sqrt{(1 - \nicefrac{1}{r})2m}$, we know from Nikiforov's
result \cite{Niki2006-walks,Niki2009jctb} that $G$ must be
a complete bipartite graph when $r = 2$, or a regular
complete $r$-partite graph when $r \geq 3$. Conversely,
each such graph is clearly $K_r^{+}[t]$-free, since
$\chi (K_r^{+}[t])=r+1$ whereas a complete $r$-partite
graph has chromatic number $r$.
\end{proof}

\subsection{Proof of Theorem \ref{thm:color-critical-I}}

With the help of Theorem \ref{thm:color-critical-when-stable}, it is straightforward to derive Theorem \ref{thm:color-critical-I}.

\begin{proof}[{\bf Proof of Theorem \ref{thm:color-critical-I}}]
Since any color-critical graph $F$ with $\chi (F)=r+1$ is a subgraph of $K_{r}^+[t]$ for some integer $t\ge 2$, it suffices to prove the theorem for $F = K_{r}^+[t]$ with $r\ge 3$ and $t\ge 2$. 
Indeed, every $F$-free graph is
$K_r^{+}[t]$-free, and conversely every regular complete
$r$-partite graph is $F$-free as $\chi (F)=r+1$. 
Let $\varepsilon_{t,r}>0$ and $n_{t,r}$ be the constants provided by Theorem \ref{thm:color-critical-when-stable},
and set 
${\varepsilon := \min \bigl\{ 0.01, 
\frac{1}{2} \varepsilon_{t,r} \bigr\} }$.  
Let $G$ be a $K_{r}^+[t]$-free graph with $m$ edges and  
$$\lambda^2(G) \geq \left(1 - \frac{1}{r}\right) \cdot 2m.$$
Applying the edge-spectral stability Theorem \ref{thm:nikiforov-stability} with this $\varepsilon$, for  sufficiently large $m$, there exists a vertex set $C \subseteq V(G)$ such that $d(G, T_{C, r}) \leq \varepsilon m$. 
For convenience, we denote $|C|=n$. In other words, we can find disjoint vertex sets $V_1, \dots, V_r$ inside $G$ with $\lfloor \frac{n}{r} \rfloor \le |V_i| \leq \lceil \frac{n}{r} \rceil$ such that $G$ differs from $K_{V_1, \dots, V_r}$ in at most $\varepsilon m$ edges. We first record two simple claims. 

\begin{claim} \label{claim-bound-m}
    We have $(1 - \varepsilon) \left(1-\frac{1}{r} \right) \frac{n^2}{2} \le m \le (1 + 2\varepsilon) \left(1-\frac{1}{r} \right) \frac{n^2}{2}$. 
\end{claim}

\begin{poc}
   Since $T_{C,r}$ can be obtained from $G$ by changing (adding or deleting) at most $\varepsilon m$ edges, we have $(1-\varepsilon) m\le e(T_{C,r})\le (1+\varepsilon)m$. By Lemma \ref{lem:e-T}, we know that 
    \[  \left(1- \frac{1}{r} \right)\frac{n^2}{2} - \frac{r}{8} \le 
    e(T_{C,r})\le \left(1- \frac{1}{r} \right)\frac{n^2}{2} .  \] 
    Since $0< \varepsilon <0.01$, we have $\frac{1}{1- \varepsilon} < 1+ 2\varepsilon$ and  
    $m\le \frac{1}{1-\varepsilon} e(T_{C,r}) \le (1+2\varepsilon) \left(1-\frac{1}{r} \right) \frac{n^2}{2}$.
    On the other hand, we get 
    $ m\ge \frac{1}{1+\varepsilon}e(T_{C,r}) \ge (1 - \varepsilon) \left(1-\frac{1}{r} \right) \frac{n^2}{2}$,  
    where the last inequality is guaranteed by $\varepsilon^2 (1-\frac{1}{r}) \frac{n^2}{2} > \frac{r}{8}$. This inequality holds for sufficiently large $m$ since $n=\Omega_{r,\varepsilon}(\sqrt{m})$. 
\end{poc}

\begin{claim} \label{claim-subgraph-H}
    For every $i\in [r]$, 
let $U_i$ be a subset of $V_i$ with size $|U_i|  
= \min\limits_{k\in [r]} |V_k|.$ Then 
\[  d(K_{V_1,\ldots ,V_r}, K_{U_1,\ldots ,U_r}) 
\le 3r \sqrt{m}. \]
\end{claim}

\begin{poc}
    Write $n=r \cdot \lfloor \frac{n}{r}\rfloor +s$, where $0\le s<r$. The Tur\'{a}n graph $K_{V_1,\ldots ,V_r}$ has $s$ partite sets of size $\lfloor \frac{n}{r}\rfloor +1$, and $r-s$ partite sets of size $\lfloor \frac{n}{r}\rfloor$. 
    To transform $K_{V_1,\ldots ,V_r}$ to $K_{U_1,\ldots ,U_r}$, it suffices to delete exactly one vertex from each vertex set of size $\lfloor \frac{n}{r}\rfloor +1$. 
    This removes fewer than $sn $ edges. By Claim \ref{claim-bound-m}, we know that $n < 3\sqrt{m}$ and $sn < 3r \sqrt{m} $, as needed.   
\end{poc}

We denote $H := K_{U_1, \dots, U_r}$. 
Then $H$ is regular. 
By Claim \ref{claim-bound-m}, we have
$m\le (1+2\varepsilon)(1-\frac{1}{r})\frac{n^2}{2}$, and
hence $n=\Omega_{r}(\sqrt{m}\, )$. In particular,
 for sufficiently large $m$, 
$$\abs{U_i} = \min_{k\in [r]}|V_k| \ge
\lfloor n/r \rfloor \ge n_{t,r}.$$ 
By Claim \ref{claim-subgraph-H}, we get $d(K_{V_1,\ldots ,V_r},H)\le 3r\sqrt{m} \le \varepsilon m$
whenever $m\ge (3r/\varepsilon)^2$, and therefore  
\[ d(G, H) \leq d(G, K_{V_1, \dots, V_r}) + d(K_{V_1, \dots, V_r}, H) \leq 2\varepsilon m \le \varepsilon_{t,r} m. \] 
 By Theorem \ref{thm:color-critical-when-stable}, we conclude that 
$$\lambda^2 (G)\le \left(1- \frac{1}{r}\right) \cdot 2m,$$ where the equality holds if and only if $G$ is a regular complete $r$-partite graph.
\end{proof}

\section{Almost-bipartite graphs}
\label{sec:almost}

To prove Theorems \ref{thm:bipartite-rough} and \ref{thm-determine-M}, we need the following two lemmas, whose proofs are straightforward and left to the interested reader. Recall that $\mathcal{A}_F$ is the family of subgraphs induced by $\overline{I}=V(F)\setminus I$, where $I$ runs over all maximal independent sets of $F$. 
The family $\mathcal{M}_F$ consists of all $\mathcal{A}_F$-free graphs. 

  \begin{lemma} \label{fact:less-vertices}
     For an almost-bipartite graph $F$, 
every graph in $\mathcal{M}_F$ is either empty, or a non-empty graph of order less than $|F|$. There are finitely many non-empty graphs in $\mathcal{M}_F$. 
 \end{lemma} 

\begin{lemma}
\label{lem:bipartite-free}
Suppose that $F$ is an almost-bipartite graph, $b$ is an integer with $b \ge \abs{F}$ and $H$ is a graph. Then $H \vee bK_1$ is $F$-free if and only if $H \in \mathcal{M}_F$.
\end{lemma}

\subsection{Proof of Theorem \ref{thm:bipartite-rough}}

Let $G$ be an $F$-free $m$-edge graph with maximum spectral radius over all $m$-edge $F$-free graphs. 
 Since $F$ is not a star, $K_{1, m}$ is clearly $F$-free. So we have 
 \begin{equation}
     \label{eq-lower-sqrt-m}
     \lambda(G) \geq \lambda(K_{1, m}) = \sqrt{m}. 
 \end{equation}
Without loss of generality, suppose that $G$ has no isolated vertices. 
Assume that $G$ is not a complete bipartite graph. 
Our goal is to prove that $G$ satisfies 
the part (b) in Theorem \ref{thm:bipartite-rough}.

Throughout the proof, we denote $t:=|F|$.

\begin{claim}
    \label{cl:bound-A-B}
    There exists a constant $\varepsilon \in (0, 
    {10^{-8}t^{-8}}]$ such that for sufficiently large $m$, we can find disjoint vertex sets $A, B$ in $G$ such that $d(G, K_{A, B}) \leq \varepsilon m$, $|{A}| \leq 16t$ 
    and $G[A] \in \mathcal{M}_F$.
\end{claim}

\begin{poc}
    Let $\varepsilon_0$ and $n_0=16t$ be the constants in Theorem \ref{thm:color-critical-when-stable} with $r := 2$ and $t$. We denote $\varepsilon := \min\{ \varepsilon_{0}, (4n_0)^{-1}, {10^{-8}t^{-8}} \}$. By the edge-spectral stability in Theorem \ref{thm:nikiforov-stability} and Corollary \ref{cor-bi-stability}, for sufficiently large $m$, there exist two disjoint vertex sets $A, B$ in $G$ such that $d(G, K_{A, B}) \leq \varepsilon m$. We point out that $\{A,B\}$ is not necessarily a partition of the vertex set of $G$. Without loss of generality, we may assume that $\abs{A} \leq \abs{B}$.

    Since $F$ is almost-bipartite, 
    we see that $F$ is a subgraph of $K_{t, t}^+$. Then $G$ is $K_{t, t}^+$-free. We claim that $|A|<n_0$. 
    Otherwise, if $\abs{A} \geq n_0$, then Theorem \ref{thm:color-critical-when-stable} implies $\lambda(G) \leq \sqrt{m}$. 
    Combining with (\ref{eq-lower-sqrt-m}), the equality holds, and $G$ is a complete bipartite graph. This contradicts the previous assumption that $G$ is not a complete bipartite graph. Therefore, we have $\abs{A} < n_0$.
 
    Since $G$ differs from $K_{A,B}$ in at most $\varepsilon m$ edges, there are at most $\varepsilon m$ edges of $K_{A, B}$ that are not in $G[A, B]$. So there are at most $\varepsilon m$ vertices of $B$ that are not complete to $A$. On the other hand, we have $e(K_{A,B})\ge (1- \varepsilon)m$, and then $|B| \geq (1 - \varepsilon) m/|A| \geq \frac{m}{2n_0}$. 
    Since $\varepsilon \leq (4n_0)^{-1}$, we see that $B$ has at least $|B| - \varepsilon m \ge \frac{m}{4n_0} > |{F}|$ vertices that are complete to $A$. This means that $G[A] \vee t K_1$ is a subgraph of $G$. By Lemma \ref{lem:bipartite-free}, we have $G[A] \in \mathcal{M}_F$. 
\end{poc} 

\noindent 
{\bf Remark.} 
In particular, for a bipartite graph $F$, there is another simple proof by using 
K\H{o}v\'{a}ri--S\'{o}s--Tur\'{a}n's theorem, instead of  Theorem \ref{thm:color-critical-when-stable}. Indeed, since 
$G[A,B]$ is $F$-free, 
we get 
\[ (1-\varepsilon)m \le e(A,B) \le 
2|A|^{1-1/t} |B|+ t|A| 
\le  2|A|^{-1/t}(1+\varepsilon)m + t\sqrt{m} . \]
For sufficiently large $m$, we have 
$|A|\le 4^{t}$ and $|B|\ge 
\frac{m}{2|A|} > m/4^{t+1}$. 

\medskip 

 Let $\bm{x}=(x_i)_{i\in V(G)}$ be the unit Perron--Frobenius eigenvector of $G$, and let $\lambda = \lambda(G)$.  
 From Claim \ref{cl:bound-A-B}, we know that $|A|=O(1)$ and $|B|=\Omega (m)$. Using Lemma \ref{lem:Perron-Frobenius-stability}, we next show that $x_a=\Omega(1)$ for every $a\in A$, and 
 $x_b=O(m^{-1/2})$ for every $b\in B$. 
 
\begin{claim} \label{cl:weight-a}
    For every vertex $a\in A$, we have $x_a > 10^{-1}t^{-1/2}$. 
\end{claim}

\begin{poc} 
Let $\bm{y}$ be the unit Perron--Frobenius eigenvector of $K_{A,B}$. Then $y_{a}=1/\sqrt{2|A|}$ for every $a\in A$, and $y_b=1/\sqrt{2|B|}$ for every $b\in B$. 
Observe that $\lambda (K_{A,B}) = \sqrt{e(K_{A,B})} \ge \sqrt{(1- \varepsilon)m}$. 
     Applying Lemma \ref{lem:Perron-Frobenius-stability} to $G $ and $H = K_{A, B}$ (possibly with some isolated vertices outside $H$), we have
    $$\left(x_a - \frac{1}{\sqrt{2\abs{A}}}\right)^2 \leq \frac{8 \sqrt{\varepsilon m}}{\sqrt{(1-\varepsilon)m}} \leq 16\sqrt{\varepsilon}.$$
    In particular, 
   Claim \ref{cl:bound-A-B} gives $|A| \leq 16t$. Since $\varepsilon \le 10^{-8} t^{-8}$, the above inequality yields 
    $$x_a \geq \frac{1}{\sqrt{2|A|}} - 4 \varepsilon^{1/4} >  \frac{1}{10\sqrt{t}}.$$ 
\end{poc}

With the help of Claim \ref{cl:weight-a}, we prove the following claim. 

\begin{claim}
\label{lem:small-part-2}
Let $\varepsilon, m, A, B$ be defined as in Claim \ref{cl:bound-A-B}. We denote $C:=V(G)\setminus A$. Then $C$ is an independent set of $G$, that is, every edge of $G$ is incident to a vertex of $A$. 
\end{claim}

\begin{poc} 
    Suppose on the contrary that $C$ is not an independent set. Let $E_b$ denote the set of edges in $G$ that are not incident to a vertex in $A$, and let $V_b$ be the set of vertices incident to some edge in $E_b$. 
    Then $|E_b| \le \varepsilon m$ as $d(G,K_{A,B})\le \varepsilon m$. 
    We construct a new graph $G'$ by removing all edges in $E_b$ from $G$ and adding $\abs{E_b}$ new pendant edges to a fixed vertex $a \in A$. Observe that $G'$ has $m$ edges, and $G'$ is a subgraph of $G[A] \vee \ell K_1$ for some integer $\ell$. By Lemma \ref{lem:bipartite-free}, we know that $G'$ is $F$-free. In what follows, we show that $\lambda (G') > \lambda (G)$.   
    Roughly speaking, removing the bad edges of $E_b$ decreases $\lambda(G)$ by about  $\sqrt{2\varepsilon m} \cdot O(\varepsilon)$, while adding new pendant edges increases it by about $\Omega(m^{-1/2}) \cdot \varepsilon m$. As a result, the spectral radius of $G$ increases.
    This leads to a contradiction with the maximality of $G$.  

    Suppose on the contrary that  $\lambda(G') \leq \lambda(G)$.  For notational convenience, we denote $\lambda = \lambda (G)$. Now, we construct a new vector $\bm{y}=(y_i)_{i\in V(G')}$ by setting $y_i = x_i$ if $i \in V(G)$, and $y_i = \frac{x_a}{\lambda}$ if $i$ is one of the new pendant vertices in $G'$. By the Rayleigh quotient, we have
    $$\bm{y}^T A_{G'} \bm{y} \leq \lambda \cdot \bm{y}^T \bm{y}.$$
    It is easy to compute that
    $$\bm{y}^T A_{G'} \bm{y} = \lambda - 2\sum_{ij \in E_b} x_i x_j + \frac{2 x_a^2}{\lambda} \abs{E_b}$$
    and 
        $$\lambda \cdot \bm{y}^T \bm{y} = \lambda + \frac{x_a^2}{\lambda} \cdot \abs{E_b}.$$
    So we must have
    \begin{equation} \label{eq-2xi-xj}
     \frac{x_a^2}{\lambda} \abs{E_b} \leq 
      \sum_{ij \in E_b} 2x_i x_j . 
    \end{equation} 
    On the other hand, we have
    $$\sum_{ij \in E_b} 2x_i x_j \leq \sqrt{2\abs{E_b}} \cdot \sum_{i \in V_b} x_i^2.$$ 
    For each $i \in V_b$, let $d_i$ denote the degree of $i$ in $G$, and let $d'_i$ denote the number of edges in $E_b$ incident to $i$. 
    By the Cauchy--Schwarz inequality, for each $i \in V_b$, we have
    $$x_i^2 = \frac{1}{\lambda^2} 
    \left(\sum_{j \in N_G(i)} x_j\right)^2 \leq \frac{1}{\lambda^2} d_i \sum_{j \in N_G(i)} x_j^2 \leq \frac{d_i}{\lambda^2}.$$
 Clearly, we have $d_i \leq d'_i + |A| \leq 32t d'_i$ since $|A|\le 16t$. 
    Then 
    $$\sum_{i \in V_b} x_i^2 \leq \frac{1}{\lambda^2} \sum_{i \in V_b} 32t d_i' = \frac{64t |E_b|}{\lambda^2}.$$
    Combining the results above, we get 
    from (\ref{eq-2xi-xj}) that 
    $$ \frac{x_a^2}{\lambda} |E_b| 
    \le \sqrt{2|E_b|} \cdot \frac{64t |E_b|}{\lambda^2} .$$
    Rearranging the above inequality, we have
    $$x_a^2 \le \frac{64t\sqrt{2|E_b|}}{\lambda} \le 
    64t \sqrt{2\varepsilon},$$
    where the last inequality holds since $|E_b|\le \varepsilon m$ and $\lambda \ge \sqrt{m}$ as (\ref{eq-lower-sqrt-m}). 
   Setting $\varepsilon \le 10^{-8}t^{-8}$, we get $x_a^2 \le 10^{-2}t^{-1}$, which leads to a contradiction with Claim \ref{cl:weight-a}.  
\end{poc}

\begin{claim} \label{cl:A-less-F}
    $G[A]$ is a non-empty graph and $|A| < |F|$. 
\end{claim}

\begin{poc}
     Suppose on the contrary that $G[A]$ is empty. 
     By Claim \ref{lem:small-part-2}, we see that $C$ is an independent set of $G$. Then $G$ is a bipartite graph on $A$ and $C$. Thus, we have $\lambda (G)\le \sqrt{m}$. Recall in (\ref{eq-lower-sqrt-m})  that $\lambda (G)\ge \sqrt{m}$. It follows that $\lambda (G)=\sqrt{m}$ and $G$ is a complete bipartite graph, which contradicts the assumption. So $G[A]$ is non-empty. 
     By Lemma \ref{fact:less-vertices}, we get $|A|< |F|$. 
\end{poc}

To prove Theorem \ref{thm:bipartite-rough}, 
it suffices to show the following claim. 

\begin{claim} \label{cl:at-most-one}
    At most one vertex of $C$ is not complete to $A$.
\end{claim}  

\begin{poc}
Suppose for the sake of contradiction that two distinct vertices $u, v \in C$ are not complete to $A$.  Without loss of generality, we may assume that $d_u \geq d_v$, where $d_u$ and $d_v$ are the degrees of $u$ and $v$, respectively. 
Let $G'$ be the graph obtained by deleting an arbitrary edge adjacent to $v$, and then adding an arbitrary edge from $u$ to any vertex in $A$ not adjacent to $u$. Observe that $G'$ has $m$ edges and is $F$-free by Lemma \ref{lem:bipartite-free}. 
In the sequel, we will show that $\lambda(G') > \lambda(G)$, thereby reaching a contradiction.

 Let $\bm{x}\in \mathbb{R}^{|V|}$ be the unit Perron--Frobenius eigenvector of $G$. 
 By Claim \ref{cl:weight-a}, we have $x_a>10^{-1}t^{-1/2}>0$ for every $a\in A$. By Claim \ref{lem:small-part-2}, the neighborhoods of $u$ and $v$ are contained in $A$, and both are non-empty as $G$ has no isolated vertices. Hence, the eigen-equation at $u$ gives
 $\lambda x_u=\sum_{a\in N(u)}x_a>0$, and likewise
 $x_v>0$. In particular, we get $x_u^2+x_v^2>0$.  
 Let $N(u), N(v), N'(u), N'(v)$ be the neighborhoods of $u, v$ in $G$ and in $G'$, respectively. We denote 
$$\Gamma := x_u \cdot \sum_{a \in N(u)} x_a + x_v \cdot \sum_{a \in N(v)} x_a.$$
We construct a new unit vector $\bm{y}$ by setting $y_i = x_i$ for every $i \in V(G)\setminus \{u, v\}$, and setting $y_u, y_v$ as nonnegative real numbers to maximize 
$$\Gamma' := y_u \cdot \sum_{a \in N'(u)} x_a + y_v \cdot \sum_{a \in N'(v)} x_a $$
subject to the constraint $y_u^2 + y_v^2 = x_u^2 + x_v^2$. 
Observe that 
$$\lambda(G') - \lambda(G) \geq \bm{y}^T A_{G'} \bm{y} - \bm{x}^T A_{G} \bm{x} = 2(\Gamma' - \Gamma).$$
To show $\lambda(G') - \lambda (G)>0$, 
we can choose the values of $y_u$ and $y_v$ in order to maximize $\Gamma'$.  
By the equality case of the Cauchy--Schwarz inequality, we have
$$\Gamma' = \sqrt{\left(\sum_{a \in N'(u)} x_a\right)^2 + \left(\sum_{a \in N'(v)} x_a\right)^2} \cdot \sqrt{x_u^2 + x_v^2}. $$
On the other hand, the Cauchy--Schwarz inequality yields 
$$\Gamma \leq \sqrt{\left(\sum_{a \in N(u)} x_a\right)^2 + \left(\sum_{a \in N(v)} x_a\right)^2} \cdot \sqrt{x_u^2 + x_v^2}.$$
In order to show that $\lambda(G') - \lambda(G) > 0$, it suffices to show that
$$D := \left(\sum_{a \in N'(u)} x_a\right)^2 + \left(\sum_{a \in N'(v)} x_a\right)^2 - \left(\sum_{a \in N(u)} x_a\right)^2 - \left(\sum_{a \in N(v)} x_a\right)^2 > 0.$$
This follows from our earlier observation that the values $x_a$ for $a \in A$ are nearly identical. Recall that 
$|A||B|\ge e(A,B)\ge (1- \varepsilon )m$. 
By setting $H=K_{A,B}$ in Lemma \ref{lem:Perron-Frobenius-stability}, we have
$$\sum_{a \in A} \left(x_a - \frac{1}{\sqrt{2\abs{A}}}\right)^2 \leq \frac{8 \sqrt{\varepsilon m}}{\sqrt{\abs{A} \abs{B}}} \leq 8 \sqrt{\frac{\varepsilon}{1 - \varepsilon}} \leq 16 \varepsilon^{1/2}.$$
Thus, for any subset $N \subseteq A$, the Cauchy--Schwarz inequality implies that
$$\left| \left(\sum_{a \in N} x_a\right)  - \frac{\abs{N}}{\sqrt{2\abs{A}}} \right| \leq \abs{A}^{1/2} \cdot 4 \varepsilon^{1/4} .$$
On the other hand, we observe that 
\[ \left(\sum_{a \in N} x_a\right) + 
\frac{|N|}{\sqrt{2|A|}} \le 2\sqrt{2}\cdot 
\frac{|N|}{\sqrt{2|A|}}. 
\]
By Claim \ref{cl:A-less-F}, 
we have $|N|\le |A| < |F|$. 
Thus, it follows that 
$$\left| \left(\sum_{a \in N} x_a\right)^2 - \frac{|N|^2}{2|A|} \right| \leq 
4|A|^{1/2}  \varepsilon^{1/4} \cdot 
\frac{2|N|}{\sqrt{|A|}} < 
8|F|\varepsilon^{1/4}.$$
Therefore, we obtain
$$\left| D - \frac{(d_u + 1)^2 + (d_v - 1)^2 - d_u^2 - d_v^2}{2|A|} \right| \leq 32 |F| \varepsilon^{1/4}.$$
Since $\varepsilon\le 10^{-8}t^{-8}$ and $|F|=t$, we have $32|F|\varepsilon^{1/4}\le 32\cdot 10^{-2}t^{-1}
<(2|F|)^{-1}$. On the other hand, we have $(d_u + 1)^2 + (d_v - 1)^2 - d_u^2 - d_v^2=2(d_u-d_v)+2 \geq 2$. So we conclude that
$$D > \frac{1}{|A|} - 32|F|\varepsilon^{1/4} 
> \frac{1}{2\abs{F}} > 0$$
as desired, reaching a contradiction.
\end{poc}

\subsection{Proof of Theorem \ref{thm-determine-M}}

Theorem~\ref{thm:bipartite-rough} reduces the determination of the spectral extremal graph $G$ to a finite family of candidates: it remains to decide which graph $G[A]\in\mathcal{M}_F$ gives the largest spectral radius.
{Since $\mathcal{M}_F$ contains only finitely many
non-empty graphs, the ratios $e(M)/v(M)$ take finitely
many values. The next lemma approximates $\lambda(G)$ with an error of order $m^{-1/2}$, and thus turns the question into a finite comparison.} 
Theorem \ref{thm-determine-M} follows from 
 Theorem \ref{thm:bipartite-rough} and 
 Lemma \ref{lem-estimate}.

\begin{lemma} \label{lem-estimate}
    Let $G$ be an $m$-edge graph with 
    $V(G)=A\sqcup C$ such that $|A|=O(1)$, 
    $G[A]\cong M$, $e(C)=0$  
     and all but at most one vertex of $C$ are complete to $A$. 
    Then 
    $$\lambda(G) = \sqrt{m} + \frac{e(M)}{v(M)} + O(m^{-1/2}),$$
    where $e(M)$ and $v(M)$ denote the numbers of edges and vertices of $M$, respectively. 
\end{lemma}

\begin{proof}
Write $a:=|A|=v(M)$ and $c:=|C|$. Let $\bm{x}$ 
be the unit
Perron--Frobenius eigenvector of $G$ and
$\lambda:=\lambda(G)$. Set $S:=\sum_{u\in A}x_u$. 
If some vertex of $C$ is not complete to $A$, we
denote it by $v_0$; otherwise we let $v_0$ be an arbitrary vertex of $C$. 
In either case, we put $r:=|N(v_0)|$, so that $0\le r\le a$.  
Since $C$ is an independent set, 
and every vertex of $C\setminus \{v_0\}$ 
is complete to $A$, 
we have 
$e(G)=e(M)+a(c-1)+r$ and
\begin{equation*}
a(c-1)=m-e(M)-r=m+O(1), 
\end{equation*}
where we used $a=O(1)$ and $e(M)\le \binom{a}{2}=O(1)$.
In particular, we get $c=\Theta (m)$; we shall use $c\ge 2$
below.
Moreover, $A$ and $C\setminus\{v_0\}$ induce a supergraph
of $K_{a,c-1}$, so we get 
\begin{equation}\label{eq:lambda-rough}
\sqrt{m}-O(m^{-1/2})\le\sqrt{a(c-1)}\le\lambda
\le\sqrt{2m}.
\end{equation} 

First, we bound the coordinates of the eigenvector on vertices of $C$. 
For every vertex $v\in C\setminus\{v_0\}$, we have  $N(v)=A$, so the eigen-equation at $v$ reads $\lambda x_v= \sum_{u\in A} x_u = S$. Hence $\bm{x}$
takes the common value $S/\lambda$ on every vertex of $C\setminus\{v_0\}$. Moreover, we have $\lambda x_{v_0}\le S$.

Second, we bound the coordinates of the eigenvector on vertices of $A$.  
Let $u\in A$. Since $u$ is adjacent to every vertex of
$C\setminus\{v_0\}$, the eigen-equation at $u$ reads
\begin{equation}\label{eq:eigen-A}
\lambda x_u=(c-1)\frac{S}{\lambda}
+ 1_{u\sim v_0} \cdot x_{v_0}+\sum_{u'\in N_M(u)}x_{u'}.
\end{equation}
The first term on the right does not depend on $u$, while
the remaining two lie in $[0,a+1]$ because $x_i\le 1$ for
every vertex $i$. Hence
$|\lambda x_u-\lambda x_{u''}|\le a+1$ for all
$u,u''\in A$. Thus, by \eqref{eq:lambda-rough},
\begin{equation}\label{eq:uniform}
\Bigl|x_u-\frac{S}{a}\Bigr|
\le\max_{u''\in A}|x_u-x_{u''}|
\le\frac{a+1}{\lambda}=O\bigl(m^{-1/2}\bigr)
\quad \text{for any~}u\in A,
\end{equation}
where the first inequality holds since $S/a$ is the average of $x_u$ over $u\in A$.

Third, we show that $S$ is bounded away from zero. 
As the $x_u$ are non-negative, we have
$\sum_{u\in A}x_u^{2}\le S^{2}$. By the above discussion and
\eqref{eq:lambda-rough}, we have
\[
\sum_{v\in C}x_v^{2}\le c\Bigl(\frac{S}{\lambda}\Bigr)^{2}
\le\frac{cS^{2}}{a(c-1)}\le\frac{2S^{2}}{a}.
\]
Since $\bm{x}$ is a unit vector, i.e., $\sum_{u\in A}x_u^2 + \sum_{v\in C}x_v^2 =1$, the above two upper  bounds imply 
$1\le (1+2/a) S^{2} \le 3S^{2}$. Thus, 
we get $S\ge 1/ \sqrt{3}$. On the other hand, by the
Cauchy--Schwarz inequality, we get $(\sum_{u\in A} x_u)^2 \le (\sum_{u\in A} x_u^2) (\sum_{u\in A}1) \le a$, which yields 
 $S = \sum_{u\in A} x_u \le\sqrt{a}$. Consequently, we obtain $S=\Theta (1)$, which is a constant independent of $m$.  

Finally, summing \eqref{eq:eigen-A} over all $u\in A$,  we obtain
\begin{equation} \label{eq-iden}
\lambda S=a(c-1)\frac{S}{\lambda}+r x_{v_0}
+\sum_{u\in A}d_M(u)x_u.
\end{equation}
For the second term, 
 the previous discussion implies $x_{v_0}\le S /\lambda$ and $r x_{v_0}\le aS/\lambda
=O(m^{-1/2})$. For the third term, using \eqref{eq:uniform} and
$\sum_{u\in A}d_M(u)=2e(M)$, it follows that 
\[
\sum_{u\in A}d_M(u)x_u=2e(M)\frac{S}{a}
+O\bigl(m^{-1/2}\bigr).
\]
Therefore, the equation (\ref{eq-iden}) reduces to 
\[  \lambda S = a(c-1) \frac{S}{\lambda} + 2e(M)\frac{S}{a} + O(m^{-1/2}).  \]
Multiplying the above identity by $\lambda/S$, 
using $S= \Theta (1)$ and $\lambda = \Theta(\sqrt{m})$, we conclude that 
\begin{equation} \label{eq-conclude}
\lambda^{2}=a(c-1)+\frac{2e(M)}{a}\lambda+O(1).
\end{equation}
Combining this with $a(c-1)=m+O(1)$ and
$\lambda\le\sqrt{2m}$ yields $\lambda^{2}=m+O(\sqrt{m}\,)$,
hence $\lambda=\sqrt{m}+O(1)$. Substituting this back
into (\ref{eq-conclude}) gives
\[
\lambda^{2}=m+\frac{2e(M)}{a}\sqrt{m}+O(1)
=\Bigl(\sqrt{m}+\frac{e(M)}{a}\Bigr)^{2}+O(1).
\]
Since $\lambda\ge\sqrt{m}-O(m^{-1/2})$, taking
square roots  gives
$\lambda=\sqrt{m}+ \frac{e(M)}{a}+O(m^{-1/2})$, as desired.  
\end{proof}

The above proof uses only that the number of vertices of $C$ that are not complete to $A$ is bounded by a fixed integer $K=O(1)$, rather than by one. Indeed, since $\lambda \ge \sqrt{m} -O(m^{-1/2})$ and $\lambda x_v\le S$ for every
$v\in C$, the error terms in \eqref{eq:uniform} and (\ref{eq-iden}) become $(a+K)/\lambda$ and $Ka\,S/\lambda$, both of them are of order $O(m^{-1/2})$. So Lemma~\ref{lem-estimate} remains valid under this weaker hypothesis.  
In addition, we provide an alternative proof, 
which may be of independent interest.

\begin{proof}[Second proof of Lemma~\ref{lem-estimate}]
We first determine the spectral
radius of the join of $M$ with an independent set. 
For an integer $t\ge 1$, let $G_t:=M\vee tK_1$ and
$m_t:=e(G_t)$. For sufficiently large $m_t$, we claim that
\begin{equation}\label{eq:join}
\lambda(G_t)=\sqrt{m_t}+\frac{e(M)}{v(M)}
+O\bigl(m_t^{-1/2}\bigr).
\end{equation}
We write $a:=|A|=v(M)$. Then $m_t=e(M)+at$. 
Since $m_t$ is sufficiently large and $|A| = O(1)$, 
 we know that $t$ is sufficiently large so that $t>4a$. Let
$\lambda:=\lambda(G_t)$ and let $\bm{x}$ be the unit
Perron--Frobenius eigenvector of $G_t$. Since $G_t$ is
connected, the vector $\bm{x}$ is unique and positive.  By symmetry, $\bm{x}$ shares a common value $z>0$ on the $t$ vertices of the independent part. 
Writing $\bm{x}_M$ for the
restriction of $\bm{x}$ to $V(M)$, the eigen-equations of $G_t$ read as 
\[
\lambda z=\mathbf{1}^{\top}\bm{x}_M
\quad\text{and}\quad
(\lambda I-A_M)\,\bm{x}_M=tz\,\mathbf{1},
\]
where $A_M$ is the adjacency matrix of $M$ and
$\mathbf{1}$ is the all-ones vector. Since
$K_{a,t}\subseteq G_t$, we have
$\lambda\ge\sqrt{at}>2a>\lambda(A_M)$, so the matrix
$\lambda I-A_M$ is invertible and
$(\lambda I-A_M)^{-1}=\sum_{k\ge 0}A_M^{k}/\lambda^{k+1}$.
Substituting the second eigen-equation into the first one
and dividing by $z>0$, we obtain
\begin{equation}\label{eq:walks}
\lambda^{2}=t\sum_{k\ge 0}\frac{w_{k+1}}{\lambda^{k}},
\end{equation}
where $w_{k+1}:=\mathbf{1}^{\top}A_M^{k}\mathbf{1}$ denotes
the number of walks with $k+1$ vertices in $M$. 

Since $w_1=v(M)=a$ and $w_2=2e(M)$, 
 the first two terms of \eqref{eq:walks} are $at$ and $\frac{2e(M)t}{\lambda}$.  
Note that $w_{k+1}\le a^{k+1}$ for every $k\ge 1$. 
Using $\lambda\ge 2a$ and $\lambda^{2}\ge at$,
the tail of \eqref{eq:walks} satisfies
\[
t\sum_{k\ge 2}\frac{w_{k+1}}{\lambda^{k}}
\le \frac{ta^{3}}{\lambda^{2}}
\sum_{k\ge 0}\Bigl(\frac{a}{\lambda}\Bigr)^{k}
\le \frac{2ta^{3}}{\lambda^{2}}\le 2a^{2}=O(1).
\]
Hence \eqref{eq:walks} yields
\begin{equation}\label{eq:second}
\lambda^{2}=at+\frac{2e(M)t}{\lambda}+O(1).
\end{equation}
Since $t/\lambda\le\sqrt{t/a}$, we deduce from
\eqref{eq:second} that $\lambda^{2}=at+O(\sqrt{t})$, and
therefore $\lambda=\sqrt{at}+O(1)$ and
$t/\lambda=\sqrt{t/a}+O(1)$. Substituting back into
\eqref{eq:second}, we conclude that
\[
\lambda^{2}
=at+\frac{2e(M)}{a}\sqrt{at}+O(1)
=\Bigl(\sqrt{at}+\frac{e(M)}{a}\Bigr)^{2}+O(1).
\]
Note that $at=m_t-e(M) = m_t- O(1)$ and then 
$\sqrt{at}=\sqrt{m_t}+O(m_t^{-1/2})$.  
Thus, taking square roots of the above equality gives 
$\lambda = \sqrt{m_t} + \frac{e(M)}{a} + O(m_t^{-1/2})$, as desired in \eqref{eq:join}.

We now consider the general case of $G$, which may have a non-complete vertex in $C$.  
Let $t$ be the number of
vertices of $C$ that are complete to $A$, and let
$r \in\{0,1,\dots,a\}$ 
be the degree of the non-complete vertex of $C$ if it exists, 
and $r:=0$ otherwise. Then $m=e(G)=e(M)+at+r$. Clearly, we have
$
M\vee tK_1 \subseteq G \subseteq M\vee(t+1)K_1 
$. 
Since the spectral radius is monotone under adding edges, we get 
$
\lambda(M\vee tK_1)\le\lambda(G)
\le\lambda\bigl(M\vee(t+1)K_1\bigr).
$ 
The two outer graphs have $m-r$ and $m+(a-r)$ edges,
respectively, and both numbers differ from $m$ by at most $a$. Since $a=O(1)$, applying the estimate of \eqref{eq:join} to both sides and using
$\sqrt{m\pm a}=\sqrt{m}+O(m^{-1/2})$, we conclude that
$\lambda(G)=\sqrt{m}+e(M)/a+O(m^{-1/2})$. This completes the proof. 
\end{proof}

\section{Concluding remarks}

In this paper, we have introduced the edge-spectral stability method and provided a complete solution to the Brualdi--Hoffman--Tur\'{a}n problem for all color-critical graphs. We first showed that for any color-critical graph $F$ with $\chi (F)=r+1\ge 4$, if $m$ is sufficiently large and $G$ is an $m$-edge $F$-free graph, then $\lambda^2(G)\le (1- \frac{1}{r})2m$, where the equality holds if and only if $G$ is a regular complete $r$-partite graph; see Theorem \ref{thm:color-critical-I}.  
This result significantly extends Nikiforov's theorem \cite{Niki2002} and can be viewed as a spectral counterpart of Simonovits' theorem \cite{Sim1966}.  
Second, we provided an asymptotic formula and characterized the structure of extremal graphs for almost-bipartite graphs; see Theorems \ref{thm:bipartite-rough} and \ref{thm-determine-M}. As applications, 
we determined the unique spectral extremal graph for all sufficiently large integers $m$ when we forbid complete bipartite graphs plus an edge, cycles plus an edge, and theta graphs; see Theorems \ref{thm-unify} and \ref{thm-theta-rpq}. This settles several open problems and conjectures proposed in \cite{LZZ2024, LL2025laa}.   
Moreover, we established a general theorem for bipartite graphs; see Theorem \ref{thm-bipart}. 
In addition, we solved the edge-spectral extremal problem for $k$-planar graphs for each $k\ge 1$; see Theorem \ref{thm-k-planar}.

At the end of this paper, we conclude the paper by proposing some spectral extremal problems for interested readers. 
An $\ell$-walk is a sequence of $\ell$ vertices $v_1v_2\cdots v_{\ell}$ of $G$ such that $v_i$ is adjacent to $v_{i+1}$ for every $i=1,\ldots ,\ell -1$. Let $w_{\ell}(G)$ be the number of $\ell$-walks in $G$. 
In 2006, Nikiforov \cite{Niki2006-walks} showed that 
for every integer $\ell \ge 1$, if $G$ is a $K_{r+1}$-free graph, then 
$$\lambda^{\ell}(G)\le \Big(1-\frac{1}{r} \Big)w_{\ell}(G).$$
We refer to \cite{Chao-Yu-2024} for an extension by the entropy  method.  
This gives a unified extension of Wilf's bound and Nikiforov's bound, since $w_1(G)=n$ and $w_2(G)=2m$. Motivated by this result, the authors \cite{LLZ2025-part-1} proved a unified extension of Theorems \ref{thm-spec-ESS} and \ref{thm:LLZ2025}. 

\begin{theorem}[Li--Liu--Zhang \cite{LLZ2025-part-1}]
  \label{thm-LLZ-walks}  For any graph $F$ with chromatic number $\chi (F)=r+1\ge 3$, if $G$ is an $F$-free graph, then for every $\ell \ge 1$, we have 
  $$\lambda^{\ell}(G) \le \left( 1- \frac{1}{r} + o(1) \right)w_{\ell}(G).$$ 
\end{theorem}

Inspired by Theorems \ref{thm:color-critical-I} and \ref{thm-LLZ-walks}, it is natural to ask whether the
error term $o(1)$ in Theorem \ref{thm-LLZ-walks} can be removed when $F$ is color-critical. We propose the
following conjecture. 

\begin{conj}
    Let $F$ be a color-critical graph with $\chi (F)=r+1\ge 4$. If $G$ is an $F$-free graph, then for every $\ell \ge 1$ and sufficiently large $w_{\ell}(G)$, we have 
    $$\lambda^{\ell}(G) \le \left( 1- \frac{1}{r}  \right)w_{\ell}(G).$$ 
\end{conj}

In 2007,  Bollob\'{a}s and Nikiforov \cite{BN2007jctb} 
proposed the following conjecture, 
which bounds the sum of the square of the two largest eigenvalues of a graph with bounded clique number.

\begin{conj}[Bollob\'{a}s--Nikiforov \cite{BN2007jctb}] \label{conj-BN}
Let $G$ be a $K_{r+1}$-free graph of order at least $r+1$ 
with $m$ edges. Then 
$\lambda_1^2(G)+ \lambda_2^2(G) 
\le 2m\left( 1-\frac{1}{r}\right)$. 
\end{conj}

In 2021, Lin, Ning and Wu \cite{LNW2021} 
confirmed the base case $r=2$.  
The third author \cite{Zhang2024} confirmed this conjecture when $G$ is a regular graph. 
The general case was asymptotically confirmed by Coutinho, Spier and the third author  \cite{CSZ2024} for graphs with large clique number. 
Motivated by our Theorem \ref{thm:color-critical-I} and Conjecture \ref{conj-BN}, we propose the following conjecture. 

\begin{conj} 
Let $F$ be a color-critical graph with $\chi (F)=r+1\ge 4$. For sufficiently large $m$, 
if $G$ is an $F$-free graph with $m$ edges, then 
$\lambda_1^2(G)+ \lambda_2^2(G) 
\le 2m\left( 1-\frac{1}{r}\right)$. 
\end{conj}

\section*{Acknowledgements} 
The authors are grateful to the anonymous referee for the careful reading of the paper and for many valuable comments, which improved both the accuracy and the presentation of our results.   
While this paper was being written, 
Fang, Lin and Zhai \cite{FLZ2025-bipartite} independently obtained the same asymptotics as our Theorem \ref{thm-determine-M} for a finite family of bipartite graphs by a different method. 
 Their method also applies to study the Brualdi--Hoffman--Tur\'{a}n problem for a 
much larger class of bipartite graphs. 
This paper is dedicated to Vladimir Nikiforov, whose beautiful work in spectral extremal graph theory has inspired the authors. 
This work was initiated after the 1st IBS ECOPRO student research program in fall 2023. The first and third authors would like to thank the ECOPRO group for hosting them. 
The authors also thank Prof. Hao Huang for motivating this project, Dr. Licheng Zhang for pointing us to references for $1$-planar graphs, 
and Dr. Loujun Yu and Dr. Jian Zheng for carefully reading an early draft of the paper.

\frenchspacing


\end{document}